\documentclass[11pt,oneside]{article}
\usepackage[b5paper,top=1.2in,left=0.9in]{geometry}
\usepackage{verbatim}
\usepackage{amsthm}
\usepackage{amssymb}
\usepackage{amsmath}
\usepackage{graphicx}
\usepackage{appendix}
\usepackage{enumerate}

\renewcommand{\d}{\mathrm{d}}
\newcommand{\D}{\mathrm{D}}

\newcommand{\e}{\mathrm{e}}

\newtheorem{Thm}{Theorem}[section]
\newtheorem{Lem}[Thm]{Lemma}
\newtheorem{Prop}[Thm]{Proposition}
\newtheorem{Cor}[Thm]{Corollary}
\newtheorem{Rem}[Thm]{Remark}
\newtheorem{Def}[Thm]{Definition}

\newtheorem{Fact}[Thm]{Fact}
\newtheorem{Nota}[Thm]{Notation}

\def\R{\mathbb{R}}
\def\Q{\mathbb{Q}}

\def\C{\mathbb{C}}
\def\Z{\mathbb{Z}}

\def\to{\longrightarrow}

\def\cA{\mathcal{A}}
\def\cB{\mathcal{B}}

\def\cD{\mathcal{D}}
\def\cE{\mathcal{E}}
\def\cF{\mathcal{F}}
\def\cH{\mathcal{H}}
\def\cK{\mathcal{K}}

\def\cM{\mathcal{M}}

\def\cU{\mathcal{U}}

\def\cW{\mathcal{W}}
\def\a{\alpha}
\def\b{\beta}
\def\e{\epsilon}
\def\G{\Gamma}
\def\c{\gamma}
\def\D{\Delta}
\def\d{\delta}

\def\h{\theta}
\def\l{\lambda}

\def\s{\sigma}
\def\t{\tau}

\def\w{\omega}
\def\ze{\zeta}

\def\sl{\mathfrak{sl}}
\def\g{\mathfrak{g}}

\def\ox{\otimes}
\def\o+{\oplus}
\def\bo+{\bigoplus}
\def\x{\times}
\def\p[#1,#2]{\phi_{#1,#2}}
\def\til[#1]{\widetilde{#1}}
\def\what[#1]{\widehat{#1}}

\def\Re{\mathrm{Re}}
\def\Im{\mathrm{Im}}
\def\Res{\mathrm{Res}}
\def\z[#1]{z_{#1}}

\def\oo{\infty}
\def\=>{\Longrightarrow}

\def\=={\equiv}
\def\<{\langle}
\def\>{\rangle}
\def\^{\wedge}
\def\+{\dagger}

\def\sub{\subset}
\def\inv{^{-1}}
\def\dis{\displaystyle}
\def\over[#1]{\overline{#1}}
\def\vec[#1]{\overrightarrow{#1}}
\def\mat[#1, #2]{\left[\begin{array}{ccccc}#1\end{array}\left|\begin{array}{c}#2\end{array}\right.\right]}
\def\xto[#1]{\xrightarrow{#1}}
\def\dd[#1,#2]{\frac{d#1}{d#2}}
\def\del[#1,#2]{\frac{\partial #1}{\partial #2}}
\def\Facts[#1]{\begin{Fact}\mbox{}\begin{itemize}#1\end{itemize}\end{Fact}}
\def\Notation[#1]{\begin{Nota}\mbox{}\begin{itemize}#1\end{itemize}\end{Nota}}
\def\Eqn[#1]{\begin{eqnarray*}#1\end{eqnarray*}}
\def\tab{\;\;\;\;\;\;}

\newcommand{\veca}[2][cccccccccccccccccccccccccccccccccccccccccc]{\left(\begin{array}{#1}#2 \\ \end{array} \right)}
\newcommand{\ffloor}[2][cccccccccccccccccccccccccccccccccccccccccc]{\left\lfloor\begin{array}{#1}#2\end{array}\right\rfloor}
\newcommand{\cceil}[2][cccccccccccccccccccccccccccccccccccccccccc]{\left\lceil\begin{array}{#1}#2\end{array}\right\rceil}
\newcommand{\Eq}[1]{\begin{align}#1\end{align}}
\begin{document}

\title{The classical limit of representation theory of the quantum plane}

\author{  Ivan Chi-Ho Ip\footnote{
        	Kavli IPMU (WPI), 
		The University of Tokyo, 
		Kashiwa, 
		Chiba 277-8583, 
		Japan\newline
          Email: ivan.ip@ipmu.jp}
          }

\date{\today}

\numberwithin{equation}{section}

\maketitle

\begin{abstract}
We showed that there is a complete analogue of a representation of the quantum plane $\cB_q$ where $|q|=1$,
with the classical $ax+b$ group. We showed that the Fourier transform of the representation of $\cB_q$
on $\cH=L^2(\R)$ has a limit (in the dual corepresentation) towards the Mellin transform of the unitary
representation of the $ax+b$ group, and furthermore the intertwiners of the tensor products representation has a limit towards the intertwiners of the Mellin transform of the classical $ax+b$
representation. We also wrote explicitly the multiplicative unitary defining the quantum $ax+b$ semigroup
and showed that it defines the corepresentation that is dual to the representation of $\cB_q$ above, and also correspond precisely to the classical family of unitary representation of the $ax+b$ group.
\end{abstract}

\newpage
\tableofcontents

\section{Introduction}\label{sec:intro}

The $ax+b$ group is the group of affine transformations on the real line $\R$. Together with the three dimensional Heisenberg group they can be viewed as the simplest examples of non-abelian non-compact Lie group. Various difficulties in studying higher dimensional non-compact Lie group are reflected in these simple examples. For example, in the $ax+b$ group, the unitary irreducible representations are now infinite dimensional, and the Mellin transform is used to ``diagonalize" the representation. The matrix coefficients in this case are realized as integral transformations, which can be viewed as the matrix elements with respect to a continuous basis of the representation space. These matrix elements are expressed in terms of the gamma function $\G(x)$. We will see that in the quantum picture, its $q$-analogue, the $q$-gamma function $\G_q(x)$, is closely related to the important quantum dilogarithm function $G_b(x)$. Furthermore, to deal with non-compactness, there is a need to introduce the language of multiplier $C^*$ algebra to define a natural coproduct on the algebra of continuous functions vanishing at infinity, and also to construct the non-compact Haar measure \cite{VD}. Motivating from this, in the quantum picture we must deal with unbounded operators, and the theory of functional calculus for self-adjoint operators will be the main technical tool.

The quantum plane $\cB_q$ is the Hopf *-algebra over $\C$ with \emph{self-adjoint} generators $A,B$ satisfying
\Eq{AB=q^2 BA,}
with the coproduct given by
\Eq{\D(A)=A\ox A,\tab \D(B)=B\ox A+1\ox B.}
It is known that this object is self-dual, so that they can be considered both as the quantum counterpart of $C(G)$, a certain algebra of functions on $G$, the '$ax+b$' group, or $U(\g)$, the enveloping algebra of the Lie algebra $\g$ of $G$. Classically for a Lie group $G$, $U(\g)$ and $C(G)$ are paired by treating $U(\g)$ as left invariant differential operators on $G$ and evaluate the result at the identity. In such a way, representation of $U(\g)$ on a vector space $\cH$ corresponds to corepresentation of the group algebra $C(G)$ on $\cH$ by this pairing. Therefore in order to study the quantum counterpart of these representations, naturally we would like to study the representation of the quantum plane $\cB_q$, and the corepresentation of its dual object, called $\cA_q$ in this paper, under a natural pairing.

Recently in \cite{FK}, Frenkel and Kim derived the quantum Teichm\"uller space, previously constructed by Kashaev \cite{Ka2} and by Fock and Chekhov \cite{CF}, from tensor products of a single canonical representation of the modular double of the quantum plane $\cB_q$. The representation is realized as positive unbounded self-adjoint operators acting on $\cH=L^2(\R)$, and the main ingredient in their construction of the quantum Teichm\"uller space is the decomposition of the tensor product of two $\cB_q$-representations into a direct integral parametrized by a ``multiplicity" module $M\simeq L^2(\R)$, namely:
\Eq{\cH\ox\cH\simeq M\ox \cH.}
The intertwiner of this decomposition is given by a certain kind of ``quantum dilogarithm transform" (cf. Proposition \ref{QDTrans}), where the remarkable quantum dilogarithm function has been introduced by Faddeev and Kashaev \cite{FKa}.

On the other hand, in order to define a corepresentation on the dual object $\cA_q$ with positive generators, the
space of ``continuous functions vanishing at infinity" for the quantum plane $C_\oo(\cA_q)$ based on the functional
calculus of self-adjoint operators is introduced. This coincides with Woronowicz's construction of the quantum '$ax+b$'
group \cite{WZ} using the theory of multiplicative unitaries, restricted to the semigroup setting with $B>0$, so that we don't
run into the difficulty of the self-adjointness of the coproduct. The multiplicative unitary involved produces the
corepresentation of the quantum plane desired, and the corepresentation obtained in this way is shown to have a
classical limit towards the unitary representation for the classical group. Furthermore a pairing between the dual
space corresponds to the canonical representation of $\cB_q$ by unbounded self-adjoint operators defined in \cite{FK}
mentioned above.

The modular double of the quantum plane also naturally arises in this setting. The representation of $\cB_q$ on $\cH=L^2(\R)$ only becomes algebraically irreducible when we consider also its modular double $\cB_{q\til[q]}$, so that it generates a von Neumann algebra of Type I factor, while representation of $\cB_q$ itself generates Type II$_1$ factor which is more exotic \cite{Fa}. Therefore what we are considering in this paper should be viewed as restriction of the representation on $\cH$ to $\cB_q\sub \cB_{q\til[q]}$, especially useful in studying the classical limit. On the other hand, in the dual picture, quite interestingly the modular double elements are also involved in the definition of $C_\oo(\cA_q)$ due to the analytic properties of the Mellin transform, see Remark \ref{modulardouble}.

The quantum dilogarithm function played a prominent role in this quantum theory. This function and its many variants are being studied \cite{Go, Ru, Vo} and applied to vast amount of different areas, for example the construction of the '$ax+b$' quantum group by Woronowicz et.al. \cite{PW,WZ}, the harmonic analysis of the non-compact quantum group $U_q(\sl(2,\R))$ and its modular double \cite{BT,PT1,PT2}, the $q$-deformed Toda chains \cite{KLS} and hyperbolic knot invariants \cite{Ka1}. One of the important properties of this function is its invariance under the duality $b\leftrightarrow b\inv$ that provides the basis for the definition of the modular double of $U_q(\sl(2,\R))$ first introduced by Faddeev \cite{Fa}, and also related, for example, to the self-duality of Liouville theory \cite{PT1} that has no classical counterpart.

It is an interesting problem to find a classical limit to these quantum theories described by the quantum dilogarithm function. Due to the duality between $b\leftrightarrow b\inv$ and the appearance of the term $Q=b+b\inv$, there is no classical limit by directly taking $b\to 0$. In this paper, by utilizing the properties of the quantum dilogarithm function $G_b(x)$, we showed that under a suitable rescaling of parameters and a limiting process that takes $q\to1$ from inside the unit circle in the complex plane, it is possible to obtain the classical gamma function. More precisely, by taking $b$ away from the real axis, Theorem \ref{limit} states that the following limit holds for $b^2=ir\to i0^+$:
\Eq{\lim_{r\to 0} \frac{(2\pi b)G_b(bx)}{(-2\pi ib^2)^x} = \G(x),}
where $(-2\pi ib)>0$, hence the denominator is well-defined. This gives another proof of a similar limit first observed in \cite{Ru1}.

In this way, most properties of this special function reduce to its classical analogues. For example, the $q$-binomial theorem (Lemma \ref{qbi}) derived in \cite{BT}:
\Eq{(u+v)^{it}=b\int_{C}d\t \veca{t\\\t}_b u^{i(t-\t)}v^{i\t}}
is actually the $q$-analogue of the classical formula
\Eq{(x+y)^{it}=\frac{1}{2\pi}\int_{-\oo}^\oo\frac{\G(-is)\G(-it+is)}{\G(-it)}x^{is}y^{it-is}ds,}
see Remark \ref{qbirem}. In particular, the main results of this paper state that the intertwiners of the tensor product decomposition $\cH\ox\cH\simeq \cM\ox \cH$ of the representation of $\cB_q$ given by \cite{FK} has a nice classical analogue, namely the intertwiners of the classical '$ax+b$' group representation under suitable transformation (Theorem~\ref{interlimit}):
\Eq{b^2\cF\ffloor{b\l&bt\\bt_1&bt_2}_*\to\ffloor{\l&t\\t_1&t_2}_{classical}\\
b^2\cF\cceil{b\l&bt\\bt_1&bt_2}_*\to\cceil{\l&t\\t_1&t_2}_{classical}}
as $b=ir\to i0^+$. Furthermore, the corepresentation constructed using the multiplicative unitary also has a classical limit towards the unitary representation $R_+$ of the classical $ax+b$ group (Theorem~\ref{classicalcorep}).

The study of the relationship between the quantum plane and the classical $ax+b$ group is important as it serves as building blocks towards higher quantum group. First of all, we choose to work with quantum semigroup (representing the generators by \emph{positive} operators) since it induces the $b\leftrightarrow b\inv$ duality for $SL_q^+(2,\R)$ as explained in \cite{PT1}, and it also provides an important results on the closure of tensor product of $U_q(\sl(2,\R))$ representations \cite{PT2}. These observations are essential to the relationship between quantum Liouville theory and quantum geometry on Riemann surface \cite{T}. Moreover, it is fundamental in the construction of $GL_q^+(2,\R)$ by the Drinfeld's double construction proposed in \cite{Ip}, an analogue of the classical Gauss decomposition, which provides an important first step in the study of the harmonic analysis of split real quantum groups in the case $|q|=1$ \cite{FI, Ip2}.

The present paper is organized as follows. In Section \ref{sec:ax+b} we recall the definitions and facts about the classical '$ax+b$' group and its representations, and derive the tensor product decomposition of two irreducible representations. In Section \ref{sec:special} we recall some properties of the $q$-special functions, in particular a version of the quantum dilogarithm $G_b(x)$ introduced in \cite{PT2}, and derive a special limiting procedure that enables us to compare it with the classical gamma function. In Section \ref{sec:intertwiner} we recall the $q$-intertwiner for the representation of the quantum plane $\cB_q$ that is obtained in \cite{FK} to deal with the quantization of Teichm\"uller space, and we showed in Section \ref{sec:mainresult} that this intertwiner, under suitable modification, has a classical limit towards precisely the intertwiner of the $ax+b$ group. Finally in Section \ref{sec:corep} we introduce on the dual space $\cA_q$ the space of continuous functions vanishing at infinity $C_\oo(\cA_q)$, and starting from Woronowicz's multiplicative unitary of the quantum '$ax+b$' semigroup, we derive explicitly the corepresentation of the dual space $\cA_q$. We showed that this corepresentation has a limit towards the classical $ax+b$ group representation, and on the other hand, it induces the same representation of $\cB_q$ under a non-degenerate pairing.

\textbf{Acknowledgments.} I would like to thank my advisor Professor Igor Frenkel for proposing the project and providing useful insights to the general picture of the theory. I would also like to thank Hyun Kyu Kim and Nicolae Tecu for helpful discussions.

\section{Classical $ax+b$ Group}\label{sec:ax+b}

\subsection{Representation}\label{sec:ax+b:rep}
First let us recall the theory of representation of the $ax+b$
group. The classical $ax+b$ group is by definition, the group of affine transformations on the real line $\R$, where $a>0$ and $b\in \R$, and they can be represented by a matrix of the form

\begin{eqnarray}
g(a,b)=\veca{a&b\\0&1},
\end{eqnarray}
with multiplication given by
\Eq{g(a_1,b_1)g(a_2,b_2)=\veca{a_1a_2&a_1b_2+b_1\\0&1}.}

 We will also consider the representation of the transpose group
\Eq{ g(a,c)=\veca{a&0\\c&1},}
where the multiplication is given by
\Eq{g(a_1,c_1)g(a_2,c_2)=\veca{a_1a_2&0\\c_1a_2+c_2&1}.}

This corresponds to the coproduct of the quantum plane $\cB_q$ introduced later on
(cf. Section \ref{sec:intertwiner}).

\begin{Thm}[Gelfand]\cite[Ch.V.1]{Vi} Every irreducible unitary representation of the $ax+b$ group is equivalent to one of the following (acting on the left):
\begin{itemize}
\item $R_+:=R_{-i}$ or $R_-:=R_{i}$ where $R_\l$ denote the representation of the $ax+b$ group on $L^2(\R_+, \frac{dx}{x})$ by
\Eq{R_\l(g)\cdot f(x) = e^{\l b x}f(ax);}
\item $T_\rho$, the representation on $\C$ by multiplication by $a^{i\rho}$.
\end{itemize}
\end{Thm}

Similarly, the left action of the transpose group is given by the
action of the inverse element

\Eq{ g\inv=\veca{a\inv&-\frac{c}{a}\\0&1}}

\Eq{R_\l(g^T)\cdot f(x)=e^{-\l c x/a}f(a\inv x) = R_\l(g\inv)\cdot
f(x).}

Let us recall the method of Mellin transform, which gives us an explicit expression of the matrix coefficients in terms of the gamma function:

\begin{Thm}Let $f(x)$ be a continuous function on the half line $0<~x<~\oo$. Then its Mellin transform is defined by
\Eq{\phi(s):=(\cM f)(s)=\int_0^\oo x^{s-1} f(x)dx,} whenever the
integration is absolutely convergent for $a<\Re(s)<b$. By the Mellin
inversion theorem, $f(x)$ is recovered from $\phi(s)$ by
\Eq{f(x):=(\cM\inv\phi)(x)=\frac{1}{2\pi}\int_{c-i\oo}^{c+i\oo}x^{-s}\phi(s)ds}
where $c\in\R$ is any value in between $a$ and $b$.
\end{Thm}

Here we also record some analytic properties for the Mellin transform. For further details see \cite{PK}.
\begin{Prop}\label{MTpole}
(Strip of analyticity) If $f(x)$ is a locally integrable function on $(0,\oo)$ such that it has decay property:
\Eq{\label{growth}f(x)=\left\{\begin{array}{cc}O(x^{-a-\e})&x\to 0^+,\\O(x^{-b+\e})&x\to +\oo, \end{array}\right.} for
every $\e>0$ and some $a<b$, then the Mellin transform defines an analytic function $(\cM f)(s)$ in the strip
$$a<\Re(s)<b.$$

(Analytic continuation) Assume $f(x)$ behaves algebraically for $x\to 0^+$, i.e. \Eq{f(x)\sim \sum_{k=0}^\oo A_k
x^{a_k}} where $\Re(a_k)$ increases monotonically to $\oo$ as $k\to\oo$. Then the Mellin transform $(\cM f)(s)$ can be
analytically continued into $\Re(s)\leq a=-\Re(a_0)$ as a meromorphic function with simple poles at the points $s=-a_k$
with residue $A_k$.

A similar analytic property holds for the continuation to the right half plane.

(Growth) Let $f(x)$ be a holomorphic function of the complex variable $x$ in the sector $-\a<\arg x<\b$ where
$0<\a,\b\leq \pi$, and satisfies the growth property \eqref{growth} uniformly in any sector interior to the above
sector.

Then $(\cM f)(s)$ has exponential decay in $a<\Re(s)<b$ with
\Eq{(\cM f)(s)=\left\{\begin{array}{cc}O(e^{-(\b-\e)t})&t\to
+\oo,\\O(e^{(\a-\e)t})&t\to -\oo, \end{array}\right.} for any $\e>0$
uniformly in any strip interior to $a<\Re(s)<b$.

(Parseval's formula) We have \Eq{\int_0^\oo f(x)g(x)x^{z-1}dx =
\frac{1}{2\pi i}\int_{c-i\oo}^{c+i\oo} \cM f(s)\cM g(z-s)ds,} where
$\Re(s)=c$ lies in the common strip for $\cM f$ and $\cM g$. In
particular we have \Eq{\int_0^\oo |f(x)|^2 dx =
\frac{1}{2\pi}\int_{-\oo}^\oo |\cM f(\s+it)|^2dt.}
\end{Prop}

Throughout the paper, we will restrict to a special class of functions that is dense in $L^2(\R)$.
\begin{Def}\label{cW} Let $\cW$ denote the finite $\C$-linear combinations of functions of the form
\Eq{e^{-Ax^2+Bx}P(x)} where $P(x)$ is a polynomial in $x$, $A\in
\R_{>0}$ and $B\in\C$.
\end{Def}

\begin{Prop}\label{MTprop} We have the following properties for $\cW$:
\begin{enumerate}[(a)]
\item Every function $f(z)\in \cW$ is entire analytic in $z$, and $F_y(x):=f(x+iy)$ is of rapid decay in $x$.

\item The space $\cW$ is closed under Fourier transform.

\item $\cW$ is dense in $L^2(\R)$.

\item \cite[Lemma 7.2]{Sch} $\cW$ is a core for the unbounded operator $e^{\a x}$ and $e^{\b p}$ on $L^2(\R)$ where $\a,\b\in\R$ and $p=\frac{1}{2\pi i}\frac{d}{dx}$.
\end{enumerate}
\end{Prop}

Under the Mellin transform, the representations $R_\l$ can be expressed by the following:
\begin{Prop} \cite[V.1]{Vi} The action of the $ax+b$ group on $\cW\sub L^2(\R)$ is given by
\Eq{ R_\l(g)F(w)=\int_{\R+i0} K(w,z;g)F(z)dz} where
\Eq{K(w,z;g)=\frac{\G(iw-iz)a^{-iw}}{2\pi}\left(-\frac{\l
b}{a}\right)^{iz-iw}.}

Similarly, the left action of the transposed group will be given by

\Eq{R_\l(g^T)F(w)=\int_{\R+i0} K(w,z;g)F(z)dz} where

\Eq{K(w,z;g)=\frac{\G(iw-iz)a^{iw}}{2\pi}(\l b)^{iz-iw}.}

Here the branch of the factor is chosen so that $|\arg(-\l b)|<\pi$ and the contour of integration goes above the pole at $z=w$.
\end{Prop}

\subsection{Tensor product decomposition}\label{sec:ax+b:tensor}
Using the above expressions, we can construct explicit intertwiners for the tensor product decomposition of the irreducible representation $R_+,R_-$ and $T_\rho$:
\begin{Thm}

(a) We have \Eq{R_\pm\ox R_\pm \simeq L^2(\R^+, \frac{d\a}{\a})\ox R_\pm,}
where the unitary isomorphism is given by
\begin{eqnarray}
F(\a,x)&:=&f(\frac{\a x }{\a+1}, \frac{x}{\a+1}),\\
f(x_1,x_2)&:=&F(\frac{x_1}{x_2}, x_1+x_2).
\end{eqnarray}
(This formula also holds for $R_\l\ox R_\l$ for all $\l\in\C$.)

(b) We have \Eq{R_\pm\ox R_\mp \simeq L^2(\R_{<1}, \frac{d\a}{\a})\ox
R_\mp\o+ L^2(\R_{>1}, \frac{d\a}{\a})\ox R_\pm,} where the unitary
isomorphism is given by
\begin{eqnarray}
F(\a,x)&:=&f(\frac{\a x }{|\a-1|}, \frac{x}{|\a-1|}),\\
f(x_1,x_2)&:=&F(\frac{x_1}{x_2}, |x_1-x_2|).
\end{eqnarray}

(c) We have \Eq{R_\pm \ox T_\rho \simeq R_\pm,} where the unitary isomorphism
is given by
\begin{eqnarray}
F(w)&:=&f(w-\rho),\\
f(x)&:=&F(x+\rho)
\end{eqnarray}
in the space of the Mellin transform of $R_\pm$.
\end{Thm}

\begin{proof}
Let us prove (a) for the case $R_+$, while the case for $R_-$ is similar. First of all it is obvious that the maps given are inverse of each other. To check that they are intertwiners, we compare the actions on the two spaces:
\begin{eqnarray*}
R_+(g)\cdot F(\a,x)&=&R_+(g)\cdot f(\frac{\a x}{\a+1},\frac{x}{\a+1})\\
&=&e^{-ibx} f(\frac{\a ax}{\a+1},\frac{ax}{\a+1})\\
&=&e^{-ib(x_1+x_2)}f(\frac{(\frac{x_1}{x_2} a (x_1+x_2)}{\frac{x_1}{x_2}+1},\frac{a(x_1+x_2)}{\frac{x_1}{x_2}+1})\\
&=&e^{-ib x_1}e^{-ib x_2}f(ax_1,ax_2)\\
&=&R_+\ox R_+(g)\cdot f(x_1,x_2).
\end{eqnarray*}
Finally to check that it is unitary, we compute the norm after transformation:
\begin{eqnarray*}
||F(\a,x)||^2&=&\iint |f(\frac{\a x}{\a+1},\frac{x}{\a+1})|^2 \frac{dx}{x}\frac{d\a}{\a}\\
&=&\iint |f(\a x_2,x_2)|^2 \frac{dx_2}{x_2}\frac{d\a}{\a}\\
&=&\iint |f(x_1,x_2)|^2 \frac{dx_2}{x_2}\frac{d x_1}{x_1}\\
&=&||f(x_1,x_2)||^2.
\end{eqnarray*}

For (b) the argument is similar, where we split into the case $\a<1$ and $\a>1$:
\begin{eqnarray*}
R_+\ox R_-(g)\cdot f(x_1,x_2)&=&e^{-ibx_1}e^{ibx_2}f(ax_1,ax_2)\\
&=&e^{-ibx_1}e^{ibx_2}F(\frac{x_1}{x_2},a|x_1-x_2|)\\
&=&e^{-ib\frac{\a x}{|\a-1|}}e^{ib\frac{x}{|\a-1|}}F(\a,ax)\\
&=&\left\{\begin{array}{cc} e^{-ibx}F(\a, ax)& \a>1 \\e^{ibx}F(\a,ax)& \a<1 \end{array}\right..\\
\end{eqnarray*}
as required.

Finally for (c) we use the Mellin transform expression to obtain:
\begin{eqnarray*}
R_+\ox T_\rho (g)\cdot F(w)&=&\frac{a^{i\rho}}{2\pi}\int_{-\oo}^{\oo} \G(iw-iz)a^{-iw}\left(\frac{ib}{a}\right)^{iz-iw}F(z)dz\\
&=&\frac{a^{i\rho}}{2\pi}\int_{-\oo}^{\oo} \G(iw-iz-i\rho)a^{-iw}\left(\frac{ib}{a}\right)^{iz-iw+i\rho}F(z+\rho)dz\\
&=&\frac{1}{2\pi}\int_{-\oo}^{\oo} \G(ix-iz)a^{-ix}\left(\frac{ib}{a}\right)^{iz-ix}f(z)dz\\
&=&R_+(g)\cdot f(x).
\end{eqnarray*}
\end{proof}

We will focus mainly on the case $R_+\ox R_+$. Under the Mellin transform, we can rewrite the intertwiners in terms of gamma functions as follow.

\begin{Prop}\label{classical intertwiner}Let $f(\l,t)\in \cW\ox\cW\sub L^2(\R)\ox R_+$ and $f(t_1,t_2)\in \cW\ox\cW\sub R_+\ox R_+$ where $R_+=L^2(\R)$ in the Mellin transformed picture. Then the isomorphism $R_+\ox R_+\simeq  L^2(\R)\ox R_+$ can be expressed as
\begin{eqnarray}
F(\l,t)&=&\frac{1}{2\pi}\int_{C}\frac{\G(it_2-it+i\l)\G(-it_2-i\l)}{\G(-it)}f(t-t_2,t_2)dt_2\\
f(t_1,t_2)&=&\frac{1}{2\pi}\int_{C'}\frac{\G(-i\l+it_1)\G(i\l+it_2)}{\G(it_1+it_2)}F(\l,t_1+t_2)d\l
\end{eqnarray}
where $C$ is the contour going along $\R$ that goes above the poles
of $\G(-it_2-i\l)$ and below the poles of $\G(it_2-it+i\l)$, and
similarly $C'$ is the contour along $\R$ that goes above the poles
of $\G(-i\l+it_1)$ and below the poles of $\G(i\l+it_2)$.

Hence formally we can write the above transforms as integral transformations
\begin{eqnarray}
F(\l, t)&=&\iint_{\R^2} \ffloor{\l&t\\t_1&t_2}f(t_1,t_2)dt_1dt_2\\
f(t_1, t_2)&=&\iint_{\R^2}\cceil{\l&t\\t_1&t_2}F(\l,t)d\l dt
\end{eqnarray}
where
\begin{eqnarray}
\ffloor{\l&t\\t_1&t_2}&=&\frac{1}{2\pi}\d(t_1+t_2-t)\frac{\G(i\l-it_1)\G(-it_2-i\l)}{\G(-it)}\\
\cceil{\l&t\\t_1&t_2}&=&\frac{1}{2\pi}\d(t-t_1-t_2)\frac{\G(-i\l+it_1)\G(it_2+i\l)}{\G(it)}\\
&=&\overline{\ffloor{\l&t\\t_1&t_2}}\nonumber
\end{eqnarray}
\end{Prop}
\begin{proof}
We start with $$\int_{\R+ic_2}\int_{\R+ic_1} x_1^{-it_1}x_2^{-it_2}f_1(t_1)f_2(t_2)dt_1dt_2,$$
 and transform into
$$\int_{\R+ic_2}\int_{\R+ic_1} \left(\frac{\a x}{\a+1}\right)^{-it_1}\left(\frac{x}{\a+1}\right)^{-it_2}f_1(t_1)f_2(t_2)dt_1dt_2,$$
and Mellin transform back to the $(\a,t)$ space:
\begin{eqnarray}
F(\l,t)&=&\frac{1}{(2\pi)^2}\iint_{\R_+^2}\int_{\R+ic_2}\int_{\R+ic_1} x^{it-1}\a^{i\l-1}\left(\frac{\a x}{\a+1}\right)^{-it_1}\left(\frac{x}{\a+1}\right)^{-it_2}\cdot\nonumber\\
\label{MT1}&&f_1(t_1)f_2(t_2) dt_1dt_2dxd\a\\
&=&\frac{1}{(2\pi)^2}\iint_{\R_+^2}\int_{\R+ic_2}
x^{it-1}\a^{i\l-1}\left(\frac{x}{\a+1}\right)^{-it_2}\cM\inv\cdot\nonumber\\
&&f_1(\frac{\a x}{\a+1})f_2(t_2)dt_2dx d\a.\nonumber
\end{eqnarray}
From the Mellin transform properties (Proposition \ref{MTprop}), $\cM\inv
f_1(\frac{\a x}{\a+1})$ is of rapid decay in $x$. Hence the
integrand is absolutely convergent with respect to $x$ and $t_2$ and
we can interchange the order of integration in $\eqref{MT1}$ to
obtain
\begin{eqnarray*}
F(\l,t)&=&\frac{1}{(2\pi)^2}\iint_{\R_+^2}\int_{\R+ic_2}\int_{\R+ic_1} x^{it-1}\a^{i\l-1}\left(\frac{\a x}{\a+1}\right)^{-it_1}\left(\frac{x}{\a+1}\right)^{-it_2}\cdot\\
&&f_1(t_1)f_2(t_2) dt_1dxdt_2d\a\\
&=&\frac{1}{(2\pi)^2}\int_{\R_+}\int_{\R+ic_2}\int_{\R_+}\int_{\R+ic_1}  x^{it-it_1-it_2-1}\a^{i\l-it_1-1}(\a+1)^{it_1+it_2}\cdot\\
&&f_1(t_1)f_2(t_2) dt_1dxdt_2d\a\\
&=&\frac{1}{2\pi}\int_0^\oo\int_{\R+ic_2} \a^{i\l-it+it_2-1}(\a+1)^{it}f_1(t-t_2)f_2(t_2) dt_2d\a\\
\end{eqnarray*}
by the Mellin transform property.

Next from the gamma-beta integral \cite[V.1.6(7)]{Vi}, we have
\Eq{\frac{\G(w+u)\G(-u)}{\G(w)}=\int_0^\oo t^{w+u-1}(1+t)^{-w}dt,}
where $\Re(w+u)>0, \Re(u)<0$. Assuming $\l\in\R$, we see that the
integrand is absolutely convergent in $\a$ when
$$\Re(it_2-it)>0,\tab \Re(it_2)<0.$$
Hence for $c_2>0$ and $\Im(t)>c_2$, we can interchange the order of integration to obtain
\begin{eqnarray}
F(\l,t)\nonumber &=&\frac{1}{2\pi}\int_{\R+ic_2}\int_0^\oo \a^{i\l-it+it_2-1}(\a+1)^{it}f_1(t-t_2)f_2(t_2) d\a dt_2\\
&=&\frac{1}{2\pi}\int_{\R+ic_2}\frac{\G(it_2-it+i\l)\G(-it_2-i\l)}{\G(-it)}f(t-t_2,t_2)dt_2,
\end{eqnarray}
which holds for $\Im(t)>c_2>0$.
Finally we can deform the contour of $t_2$ so that it goes under $t_2=t-\l$ and above $t_2=-\l$. Then the above expression can be analytically extended to $\Im(t)=0$, and we obtain our desired formula.

Similarly, we start with
$$\int_{\R+ic_\l}\int_{\R+ic_t}\a^{-i\l}x^{-it}F_\l(\l)F_t(t)dtd\l,$$
and transform into
$$\int_{\R+ic_\l}\int_{\R+ic_t}\left(\frac{x_1}{x_2}\right)^{-i\l}(x_1+x_2)^{-it}F_\l(\l)F_t(t) dtd\l,$$
and Mellin transform back to the $(t_1,t_2)$ space:
\begin{eqnarray}
\nonumber f(t_1,t_2)&=&\frac{1}{(2\pi)^2}\iint_{\R_+^2}\int_{\R+ic_\l}\int_{\R+ic_t}x_1^{it_1-1}x_2^{it_2-1}\left(\frac{x_1}{x_2}\right)^{-i\l}(x_1+x_2)^{-it}\\
         &&F_\l(\l)F_t(t) dtd\l dx_2dx_1\\
\nonumber&&\mbox{(replacing $x_1$ by $x_1x_2$:) }\\
\nonumber&=&\frac{1}{(2\pi)^2}\iint_{\R_+^2}\int_{\R+ic_\l}\int_{\R+ic_t}x_2^{it_1+it_2-it-1}x_1^{it_1-i\l-1}(x_1+1)^{-it}\cdot\\
\nonumber&& F_\l(\l)F_t(t)dtd\l dx_2dx_1.
\end{eqnarray}
By the same arguments, we can interchange the order of integration with respect to $d\l$ and $dx_2$, and involve the Mellin transform in $x_2$ and $t$, to obtain
$$=\frac{1}{2\pi}\int_0^\oo\int_{\R+ic_\l}x_1^{it_1-i\l-1}(x_1+1)^{-it-it_2} F_\l(\l)F_t(t_1+t_2)d\l dx_1\\$$
Finally, assuming $\t_1\in\R$, the integrand is absolutely
convergent when
$$\Re(-i\l)>0,\tab\Re(-i\l-it_2)<0.$$
Hence for $0<c_\l<-\Im(t_2)$ we can interchange the order of
integration, and obtain
\Eq{f(t_1,t_2)=\frac{1}{2\pi}\int_{\R+ic_\l}\frac{\G(-i\l+it_1)\G(i\l+it_2)}{\G(it_1+it_2)}F(\l,t_1+t_2)d\l.}
Again by shifting the contours for $\l$ so that it goes above
$\l=-t_2$ and below $\l=t_1$, the expression can be analytically
extended to $\Im(t_2)=0$, and we obtain the desired formula.
\end{proof}
These expressions will play an important role in the comparison with the quantum case.

\section{$q$-Special Functions}\label{sec:special}
\subsection{Definitions}\label{sec:special:def}
Throughout this section, we let $q=e^{\pi i b^2}$ where $b\in \R\setminus\Q$ and $0<b^2<1$, so that $|q|=1$ is not a root of unity.

We will consider the quantum dilogarithm $G_b(x)$ defined in
\cite{PT1, PT2} throughout the paper. The reason is that it admits a
nice classical limit towards the gamma function, as will be shown in
the next section, and a lot of classical formula has a straightforward $q$-analogue using $G_b(x)$, where the proofs are nearly
identical. Here we recall its definition.

Let $\w:=(w_1,w_2)\in\C^2$.

\begin{Def}The double zeta function is defined as
\Eq{\ze_2(s,z|\w):=\sum_{m_1,m_2\in\Z_{\geq0}}(z+m_1w_1+m_2w_2)^{-s}.}

The double gamma function is defined as
\Eq{\G_2(z|\w):=\exp\left(\frac{\partial}{\partial
s}\ze_2(s,z|\w)|_{s=0}\right).}

Let
\Eq{\G_b(x):=\G_2(x|b,b\inv,}
then the quantum dilogarithm is defined as the function:
\Eq{S_b(x):=\frac{\G_b(x)}{\G_b(Q-x)}.}

The following form is often useful, and will be used throughout
this paper: \Eq{G_b(x):=e^{\frac{\pi i}{2}x(x-Q)}S_b(x).}
\end{Def}

\begin{Prop} The quantum dilogarithm satisfies the following properties:

Self-duality:
\Eq{S_b(x)=S_{b\inv}(x),\;\;\;\;\;\;G_b(x)=G_{b\inv}(x).}

Functional equations: \Eq{S_b(x+b^{\pm1})=2\sin(\pi
b^{\pm1}x)S_b(x),\;\;\;\;\;\;G_b(x+b)=(1-e^{2\pi ibx})G_b(x).}

Reflection property:
\Eq{S_b(x)S_b(Q-x)=1,\;\;\;\;\;\;G_b(x)G_b(Q-x)=e^{\pi
ix(x-Q)}.}

Complex conjugation: \Eq{\overline{G_b(x)}=e^{\pi i
\bar{x}(Q-\bar{x})}G_b(\bar{x})=\frac{1}{G_b(Q-\bar{x})}.}

Analyticity:

$S_b(x)$ and $G_b(x)$ are meromorphic functions with poles at
$x=-nb-mb\inv$ and zeros at $x=Q+nb+mb\inv$, for $n,m\in\Z_{\geq0}$.

\label{asymp} Asymptotic properties:
\Eq{G_b(x)\sim\left\{\begin{array}{cc}\bar{\ze_b}&\Im(x)\to+\oo\\\ze_b
e^{\pi ix(x-Q)}&\Im(x)\to-\oo\end{array}\right.}
where
\Eq{\ze_b=e^{\frac{\pi i}{4}+\frac{\pi i}{12}(b^2+b^{-2})}.}

\label{residue} Residues:
\Eq{\lim_{x\to 0} xG_b(x)=\frac{1}{2\pi},} or more generally,
\Eq{Res\frac{1}{G_b(Q+z)}=-\frac{1}{2\pi}\prod_{k=1}^n(1-q^{2k})\inv\prod_{l=1}^m(1-\widetilde{q}^{-2l})\inv}
at $z=nb+mb\inv, n,m\in\Z_{\geq0}$ and $\widetilde{q}=e^{-\pi i
b^{-2}}$.
\end{Prop}

Let us introduce another important variant of the quantum dilogarithm function:
\Eq{g_b(x):=\frac{\bar{\ze_b}}{G_b(\frac{Q}{2}+\frac{1}{2\pi i b}\log x)}.}
\begin{Lem} Let $u, v$ be positive self-adjoint operators with $uv=q^2vu$, $q=e^{\pi i b^2}$. Then
\Eq{\label{qexp}g_b(u)g_b(v)=g_b(u+v)}
\Eq{\label{qpenta}g_b(v)g_b(u)=g_b(u)g_b(q\inv uv)g_b(v)}
\eqref{qexp} and \eqref{qpenta} are often referred to as the quantum exponential and the quantum pentagon relations.
\end{Lem}

We will also use the following useful Lemma:
\begin{Lem}\label{infprod}\cite[Prop 5]{Sh} for $\Im (b^2)>0$, $G_b(x)$ admits an infinite product description given by
\Eq{G_b(x)=\bar{\ze_b}\frac{\prod_{n=1}^\oo (1-e^{2\pi
ib\inv(x-nb\inv)})}{\prod_{n=0}^\oo(1-e^{2\pi ib(x+nb)})}.}
\end{Lem}

\begin{Lem}\label{FT} \cite[(3.31), (3.32)]{BT} We have the following Fourier transformation formula:
\Eq{\int_{\R+i0} dt e^{2\pi i t r}\frac{e^{-\pi i t^2}}{G_b(Q+i
t)}=\frac{\bar{\ze_b}}{G_b(\frac{Q}{2}-ir)}=g_b(e^{2\pi b r}),}
\Eq{\int_{\R+i0} dt e^{2\pi i t r}\frac{e^{-\pi Qt}}{G_b(Q+i
t)}=\ze_b G_b(\frac{Q}{2}-ir)=\frac{1}{g_b(e^{2\pi br})},} where the contour goes above the pole
at $t=0$.

Using the reflection properties, we also obtain \Eq{\int_{\R-i0} dt
e^{-2\pi i t r}e^{-\pi Q t}G_b(i
t)=\frac{\bar{\ze_b}}{G_b(\frac{Q}{2}-ir)},}
\Eq{\int_{\R-i0} dt
e^{-2\pi i t r}e^{\pi it^2}G_b(it)=\ze_b G_b(\frac{Q}{2}-ir),}
where
the contour goes below the pole at $t=0$.
\end{Lem}

\begin{Lem}\label{tau} \cite[Lemma 15]{PT2}We have the tau-beta theorem:
\Eq{\int_C d\t e^{-2\pi \t \b}
\frac{G_b(\a+i\t)}{G_b(Q+i\t)}=\frac{G_b(\a)G_b(\b)}{G_b(\a+\b)},}
where the contour $C$ goes along $\R$ and goes above the poles of
$G_b(Q+i\t)$ and below those of $G_b(\a+i\t)$.
\end{Lem}

\begin{Lem}[ $q$-binomial theorem] \label{qbi}\cite[B.4]{BT}
Let $u, v$ be positive self-adjoint operators with $uv=q^2vu$. We have:
\Eq{(u+v)^{it}=b\int_{C}d\t \veca{t\\\t}_b u^{i(t-\t)}v^{i\t},}
where
\Eq{\veca{t\\\t}_b=\frac{e^{2\pi
ib^2\t(t-\t)}G_b(Q+ibt)}{G_b(Q+ibt-ib\t)G_b(Q+ib\t)}=\frac{G_b(-ib\t)G_b(ib\t-ibt)}{G_b(-ibt)},}
and $C$ is the contour along $\R$ that goes above the pole at $\t=0$
and below the pole at $\t=t$.

Similarly, for $uv=q^{-2}vu$, we have:
\Eq{(u+v)^{it}=b\int_{C}d\t \veca{t\\\t}^b u^{i\t}v^{i(t-\t)},} where
\Eq{\veca{t\\\t}^b=\frac{G_b(Q+ibt)}{G_b(Q+ibt-ib\t)G_b(Q+ib\t)},}
with the same contour $C$ as above.
\end{Lem}

\begin{Rem} When $t$ approaches $-in$ for positive integer $n$, by first shifting the contour along the poles at $\t=t+ik$ for $0\leq k\leq n$, the integration vanishes and $n+1$ residues are left, which is precisely the terms in the usual $q$-binomial formula.
\end{Rem}

\begin{Rem}\label{qbirem} The $q$-binomial theorem is actually the $q$-analogue of the classical formula \cite[(3.3.9)]{PK}:
\Eq{\frac{1}{2\pi
i}\int_{c-i\oo}^{c+i\oo}\G(s)\G(a-s)x^{-s}ds=\frac{\G(a)}{(1+x)^a},}
when $0<c<\Re(a)$. After a change of variables with $x$ replaced by
$x/y$, $a$ by $-it$, $s$ by $-is$ and a suitable shift of contour,
we obtain
\Eq{(x+y)^{it}=\frac{1}{2\pi}\int_{-\oo}^\oo\frac{\G(-is)\G(-it+is)}{\G(-it)}x^{is}y^{it-is}ds,}
where the contour separates the poles of the two gamma functions. We can easily see that under the limiting process described in the next section, the $q$-binomial theorem reduces precisely to this classical formula.
\end{Rem}

\subsection{Limits of the Quantum Dilogarithm}\label{sec:special:limit}

Recall that the $b$-hypergeometric function (slightly modified from \cite{PT2}) is defined by:
\Eq{\label{Fb} F_b(\a,\b,\c; z):=\frac{G_b(\c)}{G_b(\a)G_b(\b)}\int_C (-z)^{ib\inv \t}e^{\pi i\t^2}\frac{G_b(\a+i\t)G_b(\b+i\t)G_b(-i\t)}{G_b(\c+i\t)}d\t}
where the contour along $\R$ separates the poles of $G_b(\a+i\t)G_b(\b+i\t)$ from those of
$\frac{G_b(-i\t)}{G_b(\c+i\t)}$.

In comparison with the classical formula:
\Eq{\label{2F1}{}_2F_1(a,b,c;z)=\frac{\G(c)}{\G(a)\G(b)}\frac{1}{2\pi
}\int_{C}(-z)^{is}\frac{\G(a+is)\G(b+is)\G(-is)}{\G(c+is)}ds,}
we see therefore that there is a strong analogy between the function $G_b(z)$ and the Gamma function $\G(x)$. However, we know that there is no direct classical limit $b\to 0$ because of the factor $Q=b+\frac{1}{b}$ involved in the definitions. It turns out that we can still define certain kind of limit of the function $G_b(x)$ that enables one to compare it with the classical Gamma function $\G(x)$.

Recall that by Lemma \ref{infprod}, if $\Im(b^2)>0$, then $G_b(x)$
can be expressed as a ratio of infinite product:

$$G_b(x)=\bar{\ze_b}\frac{\prod_{n=1}^\oo (1-e^{2\pi
ib\inv(x-nb\inv)})}{\prod_{n=0}^\oo(1-e^{2\pi ib(x+nb)})}$$

or after scaling:
\Eq{
G_b(bx)=\bar{\ze_b}\frac{\prod_{n=1}^\oo (1-e^{2\pi ix}e^{-2\pi i nb^{-2})})}{\prod_{n=0}^\oo(1-e^{2\pi ib^2(x+n)})}
}

In order to take the limit, we let $b^2=ir$ for real $r>0$ (more generally for $\Re(r)>0$). With respect to $q$, this means that we are going ``inside the circle", and approach $q=1$ from the interior of the unit disk.

\begin{eqnarray*}
G_b(bx)&=&\bar{\ze_b}\frac{\prod_{n=1}^\oo (1-e^{2\pi ix}e^{-2\pi  n/r)})}{\prod_{n=0}^\oo(1-e^{-2\pi r(x+n)})}\\
&=&\bar{\ze_b}\frac{(e^{2\pi i x-2\pi/r};e^{-2\pi/r})_\oo}{(q^{2x};q^2)_\oo}.
\end{eqnarray*}

Note that under $b^2=ir$, we also have
\Eq{\bar{\ze_b}=e^{-\frac{\pi i}{4}-\frac{\pi
i}{12}(b^2+b^{-2})}=e^{-\frac{\pi i}{4}+\frac{\pi r-\pi/r}{12}},} and
that when $r\to 0^+$, the term
$$(e^{2\pi ix-2\pi/r};e^{-2\pi/r})_\oo\to 1.$$

On the other hand, the denominator resembles the $q$-Gamma function:
\Eq{\G_q(x)=\frac{(q^2;q^2)_\oo}{(q^{2x};q^2)_\oo}(1-q^2)^{-x+1}.}

For the ratio $\frac{\bar{\ze_b}}{ (q^2;q^2)_\oo}$, we have the following observation:
\begin{Lem}
\Eq{\lim_{r\to 0} \frac{\bar{\ze_b}}{\sqrt{-i}|b|
(q^2;q^2)_\oo}=\lim_{r\to 0}e^{-\frac{\pi i}{4}}\frac{e^{\frac{\pi
r-\pi/r}{12}}}{\sqrt{-i}\sqrt{r} (q^2;q^2)_\oo}=1,} 
where we denote $e^{-\frac{\pi i}{4}}$ by $\sqrt{-i}$.
\end{Lem}

\begin{proof}
We write $\eta(ir)=e^{-\frac{\pi r}{12}} (q^2;q^2)_\oo$, the
Dedekind eta function. Then from the well-known functional equation:
\Eq{\eta(-\tau\inv)=\sqrt{-i\tau}\eta(\tau),} substituting $\tau=ir$,
we have:
\begin{eqnarray*}
\eta(\frac{i}{r})&=&\sqrt{r}\eta(ir)\\
e^{-\frac{\pi}{12r}}(e^{-\frac{2\pi}{r}};e^{-\frac{2\pi}{r}})_\oo &=& e^{-\frac{\pi r}{12}} \sqrt{r} (q^2;q^2)_\oo\\
\frac{e^{\frac{\pi r-\pi/r}{12}}}{\sqrt{r}{(q^2,q^2)_\oo}}&=&(e^{-\frac{2\pi}{r}};e^{-\frac{2\pi}{r}})_\oo^{-1}
\end{eqnarray*}
and taking the limit $r\to 0^+$, we have
 $$\lim_{r\to 0}(e^{-\frac{2\pi}{r}};e^{-\frac{2\pi}{r}})_\oo = 1$$ as required.
\end{proof}

Finally, combining with the obvious limit:
\Eq{\lim_{r\to 0^+}\frac{|b|^2}{1-q^2}=\frac{1}{2\pi},}
we have the following:

\begin{Thm}\label{limit} The limit holds for $b^2=ir\to i0^+$:
\Eq{\label{Glim}\lim_{r\to 0} \frac{(2\pi b)G_b(bx)}{(-2\pi ib^2)^x} = \G(x),}
where $(-2\pi ib^2)=2\pi r>0$, hence the denominator is well-defined. The limit converges uniformly for every compact set in $\C$. This gives another proof of a similar limit first observed in \cite{Ru1}.

A similar analysis shows that
\Eq{\label{Glim2}\lim_{r\to 0} \frac{(2\pi b)G_b(Q+bx)}{(-2\pi ib^2)^{x+1}} = (1-e^{2\pi i x})\G(x+1).}
\end{Thm}

\begin{Prop} The two limits \eqref{Glim} and \eqref{Glim2} are compatible with the reciprocal relations
$$G_b(x)G_b(Q-x)=e^{\pi i x(x-Q)},$$
$$\G(x)\G(1-x)=\frac{\pi}{\sin(\pi x)}.$$
\end{Prop}
\begin{proof} We have
\begin{eqnarray*}
1&=&G_b(bx)G_b(Q-bx)e^{-\pi ibx(bx-Q)}\\
&=&\left(\frac{(2\pi b)G_b(bx)}{(-2\pi ib^2)^x}\right)\left(\frac{(2\pi b)G_b(Q-bx)}{(-2\pi ib^2)^{-x+1}}\right)\frac{(-2\pi ib^2)}{(2\pi b)^2}e^{-\pi ibx(bx-Q)}\\
&\to&\G(x)\G(1-x)(1-e^{2\pi ix})\frac{-i}{2\pi}e^{\pi ix}\\
&=&\frac{\pi}{\sin(\pi x)}\frac{e^{\pi ix}-e^{-\pi ix}}{2\pi i}\\
&=&1,
\end{eqnarray*}
where we used
$$e^{-\pi ibx(bx-Q)}=e^{-\pi ix(b^2 x -b^2 -1)}=e^{-\pi ix r(x-1)}e^{\pi ix}\to e^{\pi ix}.$$
\end{proof}

\section{$q$-Intertwiners}\label{sec:intertwiner}
We begin with the definition of the quantum plane that is used in \cite{FK}.
\begin{Def}
The quantum plane $\cB_q$ for $|q|=1$ is generated by two positive self-adjoint operators $X,Y$ such that 
$$XY=q^2 YX$$ in the sense of \cite{Sch}, i.e. 
\Eq{X^{is}Y^{it}=q^{-2st}Y^{it}X^{is},} for every $s,t\in\R$ as relations between unitary operators. The coproduct is given by
\begin{eqnarray}
\label{DX}\D X&=&X\ox X,\\
\label{DY}\D Y&=&Y\ox X+1\ox Y.
\end{eqnarray}
\end{Def}

In \cite{FK}, this is realized by $X=~e^{-2\pi b p}$ and $Y=e^{2\pi bx}$, where $p=\frac{1}{2\pi i}\del[,x]$, acting as unbounded positive self-adjoint operators on $\cH=L^2(\R)$, such that
\begin{eqnarray}
\label{HKX}X\cdot f(x)&=&f(x+ib),\\
\label{HKY}Y\cdot f(x)&=&e^{2\pi b x}f(x),
\end{eqnarray}
which is well-defined for functions in the core $\cW\sub L^2(\R)$ (cf. Definition \ref{cW}). We remark that $\cB_q$ is ``dual" to the quantum plane $\cA_q$ generated by $A,B$ defined in the Section \ref{sec:corep}, due to the different coproducts.

In the study of tensor products of representations, the operators act by the coproduct \eqref{DX}, \eqref{DY}. It was shown in \cite{FK} that there is a quantum dilogarithm transform that gives a unitary isomorphism as representations of
$\cB_q$:
\Eq{\cH_1\ox \cH_2\simeq \cM\ox \cH}
where $\cM=L^2(\R)$ is
the parametrization space (or the multiplicity module), and carries the trivial representation.

\begin{Prop}\label{QDTrans} The quantum dilogarithm transform is defined on $f,\phi \in\cW\ox\cW$ by
\begin{eqnarray}
\phi(\a,x)&=&\int_\R\int_{\R-i0} \ffloor{\a&x\\x_1 & x_2}f(x_1,x_2)dx_2dx_1,\\
f(x_1,x_2)&=&\int_{\R-i0}\int_\R \cceil{\a&x\\x_1 & x_2}\phi(\a,x)d\a dx.
\end{eqnarray}

Here the kernel is given by:
\begin{eqnarray}
\ffloor{\a&x\\x_1 & x_2}&=&e^{2\pi i\a(x-x_1)}\cE_R(x-x_1,x_2-x_1),\\
\cceil{\a&x\\x_1 & x_2}&=&e^{-2\pi i\a(x-x_1)}\cE_L(x_2-x_1,x-x_1),
\end{eqnarray} 
where
\begin{eqnarray}
\cE_R(z,w)&=&e^{2\pi izw}S_R(z-w),\\
\cE_L(z,w)&=&e^{-2\pi izw}S_L(z-w),
\end{eqnarray}
and
\begin{eqnarray}
S_R(z)&=&G(z-ia)e^{i\chi+\frac{\pi}{2}(z-ia)^2},\\
S_L(z)&=&G(z-ia)e^{-i\chi-\frac{\pi}{2}(z-ia)^2},
\end{eqnarray}
where $\chi=\frac{\pi}{24}(b^2+b^{-2})$. The contour for $x_2$ goes below
the pole at $x_2=x$, and the contour for $x$ goes below the pole at
$x=x_2$.

The integral transforms are unitary, hence they extend to the whole of $\cH_1\ox~\cH_2$ and $\cM\ox\cH$ respectively.
\end{Prop}

Here the function $G(z)=G(b,b\inv;z)$ is the Ruijsenaars's definition of the quantum dilogarithm \cite{Ru}, and is given by
\Eq{\label{RuiG}G(z)=\exp\left(i\int_0^\oo \frac{dy}{y}\left(\frac{\sin(2yz)}{2\sinh(by)\sinh(b\inv
y)}-\frac{z}{y}\right)\right).} The relation between $G(z)$ and $G_b(z)$ is given by \Eqn[
G(z)&:=&G(b,b\inv;z)\\
S_2(z|a_+,a_-)&:=&G(a_+,a_-;-iz+ia)\\
a&:=&\frac{a_++a_-}{2}\\
S_b(z)&=&1/S_2(z|b,b\inv)\\
G_b(z)&=&e^{\frac{\pi i}{2}x(x-Q)}S_b(x)\\
Q&=&b+b\inv=2a,
]
i.e.
\Eq{\label{TeschG}G(b,b\inv,x)=e^{\pi i x^2/2}e^{\pi i Q^2/8}G_b(\frac{Q}{2}-ix).}
Hence we have, in terms of $G_b(x)$:
\begin{eqnarray}
\nonumber\ffloor{\a&x\\x_1&x_2}&=&\bar{\ze}_b e^{2\pi i(x-x_1)(x_2-x_1+\a)}e^{\pi i(x_2-x)^2}e^{\pi Q(x-x_2)}G_b(ix_2-ix)\\
&=&\bar{\ze}_b \frac{e^{2\pi i(x-x_1)(x_2-x_1+\a)}}{G_b(Q+ix-ix_2)},\\
\cceil{\a&x\\x_1&x_2}&=&\ze_b e^{-2\pi
i(x-x_1)(x_2-x_1+\a)}G_b(ix-ix_2),
\end{eqnarray}
where $$\ze_b=e^{\frac{\pi i}{4}+\frac{\pi i}{12}(b^2+b^{-2})},\tab \bar{\ze}_b=e^{-\frac{\pi i}{4}-\frac{\pi
i}{12}(b^2+b^{-2})}.$$
\section{Classical Limit of $q$-Intertwiners}\label{sec:mainresult}
In this section, we will compare the quantum dilogarithm transformation and the classical $ax+b$ group intertwiners, and show that they correspond to each other under the limiting procedures suggested in Section \ref{sec:special:limit}.

\subsection{Fourier transform of the $q$ Intertwiners}\label{sec:mainresult:FT}
In order to compare with the classical case, we need to take the Fourier transform of both function spaces $\cH_1\ox \cH_2$ and $\cM\ox \cH$. In order to do this correctly, it turns out that we need to modify the kernel by

\Eq{\ffloor{\a&x\\x_1&x_2}_*:=\frac{\bar{\ze_b}e^{-\pi i(x-x_1)^2}}{G_b(\frac{Q}{2}+i\a)}\ffloor{\a&x\\x_1&x_2},}
and
\Eq{\cceil{\a&x\\x_1&x_2}_*:=\ze_b e^{\pi i(x-x_1)^2}G_b(\frac{Q}{2}+i\a)\ffloor{\a&x\\x_1&x_2}.}

The extra factors depend only on $\a$ and $(x-x_1)$, hence the integral kernels are still an intertwiner. Note that $G_b(\frac{Q}{2}+i\a)$ is unitary by the complex conjugation property, so that the intertwiners are still unitary operators.

\begin{Thm} Under the Fourier transform, the intertwining maps defined on $f,\phi\in\cW\ox\cW$ become:

\Eq{\phi(\l,t)=\int_C \frac{G_b(it_2-it+i\l)G_b(-it_2-i\l)}{G_b(-it)} e^{\pi i \l(\l-2t+2t_2)}f(t-t_2,t_2)dt_2,}
\Eq{f(t_1,t_2)=\int_{C'} \frac{G_b(-i\l+it_1)G_b(i\l+it_2)}{G_b(it)} e^{\pi i \l(\l+2t_2)}e^{-2\pi i t_1 t_2}\phi(\l,t_1+t_2)d\l,}
where $C$ is the contour going along $\R$ that goes above
the poles of $\G_b(-it_2-i\l)$ and below the poles of $\G_b(it_2-it+i\l)$, and similarly $C'$ is the contour along $\R$ that goes above the poles of $\G_b(-i\l+it_1)$ and below the poles of $\G_b(i\l+it_2)$.

Hence formally we can write the above transformations as integral transformations:

\Eq{\phi(\l, t)=\iint \cF\ffloor{\l&t\\t_1&t_2}_*f(t_1,t_2)dt_1dt_2}
\Eq{f(t_1,t_2)=\iint\cF\cceil{\l&t\\t_1&t_2}_*\phi(\l,t)d\l dt} 
where the kernels are expressed as
\begin{eqnarray}\cF\ffloor{\l&t\\t_1&t_2}_*&=&\d(t_1+t_2-t)\frac{G_b(-it_1+i\l)G_b(-it_2-i\l)}{G_b(-it)} e^{\pi i \l(\l-2t_1)},\\
\cF\cceil{\l&t\\t_1&t_2}_*&=&\d(t-t_1-t_2)\frac{G_b(-i\l+it_1)G_b(it_2+i\l)}{G_b(it)} e^{\pi i \l(\l+2t_2)}e^{-2\pi i
t_1 t_2}.\nonumber\\
\end{eqnarray}

They are still intertwiners with respect to the Fourier transformed quantum plane
\Eq{\widehat{X}=e^{2\pi bx}\tab\widehat{Y}=e^{2\pi b p}} with the same coproduct.
\end{Thm}

\begin{proof} The intertwining properties are clear, since $\what[Y]\ox \what[X]$ are commutative with respect to $t_1, t_2$,
and Fourier transformation is linear, hence it preserves the action of$$\D\what[Y]=\what[Y]\ox \what[X]+1\ox \what[Y].$$
The delta distribution explains the intertwining property for $\D\what[X]=\what[X]\ox \what[X]$ explicitly.

We will calculate the integral transform using the Fourier transform property (Lemma \ref{FT}) and tau-beta integral (Lemma \ref{tau}) repeatedly. First we take the Fourier transform of $f(t_1,t_2)$:
$$\iint_{\R^2}e^{-2\pi i t_1 x_1}e^{-2\pi i t_2 x_2}f(t_1,t_2)dt_2dt_1,$$
applying the kernel:
$$\int_\R\int_{\R-i0}\iint_{\R^2}\frac{\bar{\ze_b^2}e^{-\pi i(x-x_1)^2}}{G_b(\frac{Q}{2}+i\a)}\frac{e^{2\pi i(x-x_1)(x_2-x_1+\a)}}{G_b(Q+ix-ix_2)}e^{-2\pi i t_1 x_1}e^{-2\pi i t_2 x_2}f(t_1,t_2)dt_2dt_1dx_2 dx_1,$$
and take the Fourier transform back to the target space
\begin{eqnarray}
\nonumber\phi(\l,t)&=&\iint_{\R^2}\int_\R\int_{\R-0}\iint_{\R^2}\frac{\bar{\ze_b^2}e^{-\pi i(x-x_1)^2}}{G_b(\frac{Q}{2}+i\a)}\frac{e^{2\pi i(x-x_1)(x_2-x_1+\a)}}{G_b(Q+ix-ix_2)}\cdot\\
\nonumber&&e^{-2\pi i t_1 x_1}e^{-2\pi i t_2 x_2}e^{2\pi i t x}e^{2\pi i\l\a}f(t_1,t_2)dt_2dt_1dx_2 dx_1dx d\a.\\
\end{eqnarray}

The integrand is absolutely convergent in $t_1$ and $t_2$ because $f(t_1,t_2)\in~\cW\ox~\cW$. With respect to $x_2$, using the asymptotic properties for $G_b$, we see that the absolute value of the integrand has the growth
$$\left\{\begin{array}{cc}e^{2\pi \Im(t_2)x_2}& x_2\to -\oo\\ e^{-\pi Qx_2}e^{2\pi \Im(t_2)x_2}&x_2\to +\oo\end{array}\right..$$

Hence it is absolutely convergent for $$0<\Im(t_2)<\frac{Q}{2},$$ and we can interchange the order of integration to obtain
\begin{align*}
\phi(\l,t)=&\int_{\R^5}\int_{\R-i0}\frac{\bar{\ze_b^2}e^{-\pi i(x-x_1)^2}}{G_b(\frac{Q}{2}+i\a)}\frac{e^{2\pi i(x-x_1)(x_2-x_1+\a)}}{G_b(Q+ix-ix_2)}\cdot\\
&e^{-2\pi i t_1 x_1}e^{-2\pi i t_2 x_2}e^{2\pi i t x}e^{2\pi i\l\a}dx_2  f(t_1,t_2)dt_2dt_1dx_1dx d\a\\
\intertext{Substituting $x_2$ by $x-x_2$:}
=&\int_{\R^5}\int_{\R+i0} \bar{\ze}_b^2 \frac{e^{-\pi i(x-x_1)^2}e^{2\pi i(x-x_1)(x-x_2-x_1+\a)}}{G_b(\frac{Q}{2}+i\a)G_b(Q+ix_2)}\cdot\\
&e^{-2\pi it_1x_1} e^{-2\pi i(x-x_2)t_2}e^{2\pi it x}e^{2\pi i \l\a}f(t_1,t_2)dt_2dt_1dx_1dx_2dxd\a.
\end{align*}

The relevant exponential with respect to $x_2$ is $$e^{2\pi ix_2(x_1+t_2-x)}.$$
Using Lemma \ref{FT}, integrating over $x_2$  with $r=x_1+t_2-x-iQ/2$, the integrand becomes
$$=\bar{\ze_b}\frac{G_b(ix-it_2-ix_1)}{G_b(\frac{Q}{2}+i\a)} e^{-\pi i(x-x_1)^2}e^{2\pi i(x-x_1)(x-x_1+\a)}e^{-2\pi it_1x_1}e^{-2\pi ixt_2}e^{2\pi it x}e^{2\pi i \l\a}f(t_1,t_2)$$
$$=\bar{\ze_b}\frac{G_b(ix-it_2-ix_1)}{G_b(\frac{Q}{2}+i\a)} e^{\pi i(x-x_1)^2}e^{2\pi i(x-x_1)\a}e^{-2\pi it_1x_1}e^{-2\pi ixt_2}e^{2\pi it x}e^{2\pi i \l\a}f(t_1,t_2).$$
Now the absolute value of this integrand with respect to $x_1$ has asymptotics
$$\left\{\begin{array}{cc}e^{2\pi \Im(t_1)x_1}& x_1\to -\oo\\ e^{-\pi Qx_1}e^{2\pi (\Im(t_1)+\Im(t_2))x_1}&x_1\to +\oo\end{array}\right..$$
Hence the integral with respect to  $x_1$ is absolutely convergent when
$$\Im(t_1)>0,\tab \Im(t_1+t_2)<\frac{Q}{2}.$$
So we now have
\begin{eqnarray*}
\phi(\l,t)&=&\int_{\R^5}\bar{\ze_b}\frac{G_b(ix-it_2-ix_1)}{G_b(\frac{Q}{2}+i\a)} e^{\pi i(x-x_1)^2}e^{2\pi i(x-x_1)\a}\cdot\\
&&e^{-2\pi it_1x_1}e^{-2\pi ixt_2}e^{2\pi it x}e^{2\pi i \l\a}f(t_1,t_2)dx_1dt_2dt_1dxd\a.\\
&&\mbox{Substitute $x_1$ by $-x_1-t_2+x$:}\\
&=&\int_{\R^4}\int_{\R-i\Im(t_2)}\bar{\ze_b}\frac{G_b(ix_1)}{G_b(\frac{Q}{2}+i\a)} e^{\pi i (x_1+t_2)^2} e^{2\pi i(x_1+t_2)\a}e^{-2\pi it_1(x-t_2-x_1)}\cdot\\
&&e^{-2\pi ixt_2}e^{2\pi it x}e^{2\pi i \l\a}f(t_1,t_2)dx_1dt_2dt_1dxd\a.
\end{eqnarray*}
The relevant exponential with respect to  $x_1$ is $$e^{-2\pi i x_1(-t_1-t_2-\a)}e^{\pi i x_1^2}.$$
Hence using Lemma \ref{FT}, integrating over $x_1$ (valid since $\Im(t_2)>0$) with $r=-t_1-t_2-\a$, the integrand becomes:
$$\frac{G_b(\frac{Q}{2}+it_1+it_2+i\a)}{G_b(\frac{Q}{2}+i\a)} e^{\pi i t_2^2}e^{2\pi i t_2\a}e^{2\pi i t_1t_2}e^{-2\pi i x(t_1+t_2-t)}e^{2\pi i \l\a}f_1(t_1)f_2(t_2).$$
Now we can simplify the integration with respect to  $t_1$ and $x$ using the factor $e^{-2\pi i x(t_1+t_2-t)}$, which is just a Fourier transform and its inverse, to obtain
$$\phi(\l,t)=\iint_{\R^2}\frac{G_b(\frac{Q}{2}+it+i\a)}{G_b(\frac{Q}{2}+i\a)} e^{\pi i t_2^2}e^{2\pi i t_2\a}e^{2\pi i (t-t_2)t_2}e^{2\pi i \l\a}f(t-t_2,t_2)dt_2d\a.$$
Now the absolute value of the integrand has asymptotics
$$\left\{\begin{array}{cc}e^{2\pi \Im(t)\a}& \a\to -\oo\\ e^{-2\pi \Im(t_2)\a}&\a\to +\oo\end{array}\right..$$
Hence it is absolutely convergent when $\Im(t)>0$. We do the final interchange of order of integration and integrate
with respect to  $\a$:
\begin{eqnarray*}
\phi(\l,t)&=&\iint_{\R^2}\frac{G_b(\frac{Q}{2}+it+i\a)}{G_b(\frac{Q}{2}+i\a)} e^{\pi i t_2^2}e^{2\pi i t_2\a}e^{2\pi i (t-t_2)t_2}e^{2\pi i \l\a}f(t-t_2,t_2)d\a dt_2\\
&&\mbox{Shifting the contour of $\a$ by $\a\to\a-i\frac{Q}{2}$ we get}\\
&=&\int_{\R}\int_{\R+i0}\frac{G_b(Q+it+i\a)}{G_b(Q+i\a)} e^{\pi i t_2^2}e^{2\pi i t_2\a}e^{\pi t_2 Q}e^{2\pi i (t-t_2)t_2}e^{2\pi i \l\a}e^{\pi \l Q}\\
&&f(t-t_2,t_2)d\a dt_2.
\end{eqnarray*}
The relevant exponential for $\a$ is $$e^{-2\pi \a(-it_2-i\l)},$$
therefore using the tau-beta integral (Lemma \ref{tau}) again, the integrand becomes:
$$\frac{G_b(Q+it)G_b(-it_2-i\l)}{G_b(Q+it-it_2-i\l)} e^{\pi i t_2^2}e^{\pi t_2 Q}e^{2\pi i (t-t_2)t_2}e^{\pi \l Q}f(t-t_2,t_2).$$
Finally using the reflection property $G_b(x)G_b(Q-x)=e^{\pi i x(x-Q)}$, we obtain
$$\frac{G_b(it_2-it+i\l)G_b(-it_2-i\l)}{G_b(-it)} e^{\pi i \l(\l-2t+2t_2)}f(t-t_2,t_2).$$

Therefore we have the expression
$$\phi(\l,t)=\int_{\R+ic_2}\frac{G_b(it_2-it+i\l)G_b(-it_2-i\l)}{G_b(-it)} e^{\pi i \l(\l-2t+2t_2)}f(t-t_2,t_2)dt_2,$$
valid for $0<c_2<\frac{Q}{2}$ and $\Im(t)>0$.

By a shift of contour on $t_2$ so that it goes below the pole at $t_2=t-\l$ and above the poles at $t_2=-\l$, the expression can be analytically continued to $t\in\R$, hence we can rewrite it as
$$\phi(\l,t)=\int_{C}\frac{G_b(it_2-it+i\l)G_b(-it_2-i\l)}{G_b(-it)} e^{\pi i \l(\l-2t+2t_2)}f(t-t_2,t_2)dt_2\in \cM\ox\cH$$
with the desired contour.

Working formally, for $\cF\cceil{\l&t\\t_1&t_2}_*$, the target space
is $f(\l, t)$ and the domain space is $f(t_1,t_2)$. Since Fourier transform
of complex conjugation is the complex conjugation of the inverse
Fourier transform, $\cF\cceil{\l&t\\t_1&t_2}_*$ is just the complex
conjugation of $\cF\ffloor{\l&t\\t_1&t_2}_*$. Hence we have
\begin{eqnarray*}
\cF\left\lceil\begin{array}{cc}\l&t\\t_1&t_2\end{array}\right\rceil_*&=&\d(t_1+t_2-t)\frac{G_b(-i\l+it_1)G_b(it_2+i\l)}{G_b(it)}e^{-\pi i\l(\l-2t_1)}\cdot\\
&&e^{\pi i(-it_1+i\l)(Q+it_1-i\l)}e^{\pi i(it_2+i\l)(Q-it_2-i\l)}e^{-\pi i(it_1+it_2)(Q-it_1-it_2)}\\
&=&\d(t_1+t_2-t)\frac{G_b(-i\l+it_1)G_b(it_2+i\l)}{G_b(it)} e^{\pi i \l(\l+2t_2)}e^{-2\pi i t_1 t_2}.
\end{eqnarray*}

Alternatively we can work through the integrations as in the proof above using similar techniques of interchanging orders of integration and shifting of contours.
\end{proof}

\subsection{Classical Limit}\label{sec:mainresult:limit}
\begin{Thm}\label{interlimit} Under a suitable rescaling, as $b^2\to i0^+$, or more generally, as $q\to1$ from inside the unit disk, the quantum intertwining operator has a limit towards the classical intertwining
transformation given by Proposition \ref{classical intertwiner}.
\end{Thm}

\begin{proof}
 The contour of integration is the same for the quantum and the classical intertwining transform. Therefore it
suffices to do the limit formally for the intertwiners. First of all we need to rescale the function space
$\cH=~L^2(\R)$ by $b$ on all the variables (including the parameter $\l$), hence the kernel is now given by

\begin{eqnarray*}
b^2\cF\ffloor{b\l&bt\\bt_1&bt_2}_*&=&b^2\d(b(t_1+t_2-t))\frac{G_b(-ibt_1+ib\l)G_b(-ibt_2-ib\l)}{G_b(-ibt)} e^{\pi i
b^2\l(\l-2t_1)}\\
&=&\frac{b\d(t_1+t_2-t)}{2\pi b}\frac{(2\pi b)G_b(-ibt_1+ib\l)}{(-2\pi ib^2)^{-it_1+i\l}}\frac{(2\pi b)G_b(-ibt_2-ib\l)}{(-2\pi ib^2)^{-it_2-i\l}}\cdot\\
&&\frac{(-2\pi ib^2)^{-it_1-it_2}}{(2\pi b)G_b(-ibt_1-ibt_2)} e^{\pi i b^2\l(\l-2t_1)}\\
&&\mbox{taking the limit using Theorem \ref{limit}:}\\
&\to&\frac{\d(t_1+t_2-t)}{2\pi}\frac{\G(-it_1+i\l)\G(-it_2-i\l)}{\G(-it_1-it_2)}\\
&=&\frac{1}{2\pi}\d(t_1+t_2-t)\frac{\G(-it_1+i\l)\G(-it_2-i\l)}{\G(-it)},
\end{eqnarray*}
which is precisely the classical intertwiner
$\ffloor{\l&t\\t_1&t_2}_{classical} $.

Similarly, we have
$$b^2\cF\cceil{b\l&bt\\bt_1&bt_2}_*\to\frac{1}{2\pi}\d(t_1+t_2-t)\frac{\G(-i\l+it_1)\G(it_2+i\l)}{\G(it)}=\cceil{\l&t\\t_1&t_2}_{classical}.$$
\end{proof}
\section {Corepresentation}\label{sec:corep}

In order to compare the classical representation of the $ax+b$ group, and shed light on what kind of intertwiners the above transforms are, as explained in the introduction we need to find a corepresentation of the quantum plane  $\cA_q$ generated by positive self-adjoint elements $A,B$ with $AB=q^2BA$, $|q|$=1, dual to $\cB_q$, with the same coproduct given by
 \Eqn[
\D(A)&=&A\ox A,\\
\D(B)&=&B\ox A+1\ox B. ]
The corepresentation should possess a limit that goes to the classical representation. Since the action of $\cB_q$ above is a left action, we expect to obtain a right corepresentation of $\cA_q$.

The basic idea is to define a $C^*$-algebra $C_\oo(\cA_q)$ of ``functions vanishing at infinity" of the quantum plane $\cA_q$. The technical details are given in \cite{Ip}. Here we will briefly recall the motivation and its construction.

\subsection{Algebra of Continuous Functions Vanishing at Infinity}\label{sec:corep:cont}
Before defining $C_\oo(\cA_q)$, let's look at the classical $ax+b$ group again. Denote the group by $G$ and the positive semigroup by $G_+=\{(a,b)|a>0, b>0\}$.

Consider the restriction of a rapidly decreasing analytic function $f(a,b)$ of $G$, to the semigroup $G_+$. Then the function is continuous at $b=0$, hence it has at most $O(1)$ growth as $b\to 0^+$.

Hence using the Mellin transform we can write
\Eq{f(a,b)=\int_{-i\oo}^{i\oo}\int_{c-i\oo}^{c+i\oo}F(s,t)a^{-s}b^{-t}dtds} where $c>0$ and
\Eq{F(s,t)=\frac{1}{(2\pi)^2}\int_0^\oo\int_0^\oo f(a,b)a^{s-1}b^{t-1}dadb}
 is entire analytic with respect to $s$, and
holomorphic on $\Im(t)>0$. According to Proposition \ref{MTpole}, $F(s,t)$ has rapid decay in $s,t$ in the imaginary
direction, and can be analytically continued to $\Im(t)\leq 0$ such that it is meromorphic with simple poles. Since the
function $f(a,b)$ is analytic at $b=0$, the analytic structure of $f(a,b)$ on $b$ is given by $\sum_{k=0}^\oo A_k b^k$
for some constant $A_k$, hence according to Proposition \ref{MTpole}, $F(s,t)$ has possible simple poles at $t=-n$ for
$n=0,1,2,...$.

Therefore (changing the integration to the real axis) we conclude that

\begin{Prop}
The continuous functions of $G_+$, continuous at $b=0$ and vanishing at infinity, is given by
$$C_\oo(G)|_{G_+}=\mbox{sup norm closure of $\cA^\oo(G_+)$ }$$
where
\Eq{\cA^\oo(G_+):=\mbox{Linear span of} \left\{\int_\R\int_{\R+i0}f_1(s)f_2(t)a^{is}b^{it}dsdt\right\}}
for
$f_1(s)$ entire analytic in $s$, $f_2(t)$ meromorphic in $t$ with possible simple poles at $t\in-in$, $n=0,1,2...$, and
for fixed $v>0$, both the function $f_1(s+iv)$ and $f_2(t+iv)$ is of rapid decay.
\end{Prop}

Note that this also coincides with
$$C_\oo(G)|_{G_+}=\mbox{sup norm closure of } \left\{g(\log a)f(b)|g\in C_\oo(\R);f\in C_\oo[0,\oo)\right\}$$
where $C_\oo$ denote functions vanishing at infinity.

We can also introduce an $L^2$ norm on functions of $G_+$ given by
\Eq{\|f(a,b)\|_2=\int_\R\int_{\R+\frac{1}{2}i}|f_1(s)f_2(t)|^2dtds}
according to the Parseval's formula for the Mellin transform.

Due to the appearance of the quantum dilogarithm function $G_b(iz)$ in the expression of the corepresentation in the next section, following the same line above, we define $C_\oo(\cA_q)$ as follows.

\begin{Def} The $C_\oo(\cA_q)$ space is the (operator) norm closure of $\cA^\oo(\cA_q)$ where
\Eq{\cA^\oo(\cA_q):=\mbox{ Linear span of } \left\{\int_\R\int_{\R+i0}f_1(s)f_2(t)A^{ib\inv s}B^{ib\inv t}dsdt\right\}} for
$f_1(s)$ entire analytic in $s$, $f_2(t)$ meromorphic in $t$ with possible simple poles at
$$t=-ibn-i\frac{m}{b},\tab n,m=0,1,2,...$$
and for fixed $v>0$, the function $f_1(s+iv)$ and $f_2(t+iv)$ is of rapid decay.
To define the norm, we realize $A^{ib\inv s}f(x)=e^{2\pi is}f(x)$ and $B^{ib\inv t}f(x)=e^{2\pi ip}f(x)=f(x+1)$ as unitary operators on $L^2(\R)$.
\end{Def}

As discussed in \cite{Ip}, we can also introduce an $L^2$ norm given by
\Eq{\|f(A,B)\|_2=\int_\R\int_{\R+\frac{iQ}{2}} |f_1(s)f_2(t)|^2dtds,}
where $Q=b+b\inv$. However we will focus on the $C^*$ theory in the remaining sections.

\begin{Rem}\label{modulardouble}
The above space $\cA^\oo(\cA_q)$ can be rewritten, according to the Mellin transform, as
$$\cA^\oo(\cA_q):=\mbox{ Linear span of }\left\{ g(\log A)f(B) \right\}$$
where $g(x)$ is entire analytic in $x$ and for every fixed $v$, $g(x+iv)$ is of rapid decay in $x$; $f(y)$ is a smooth
function in $y$ of rapid decay such that it admits a Puiseux series representation
\Eq{f(y)\sim\sum_{n,m=0}^\oo \a_{mn}y^{n+m/b^2}}
 at $y=0$.

Recall that the modular double elements \cite{Fa} are given by non-integral power
$$\til[A]=A^{\frac{1}{b^2}}\tab \til[B]=B^{\frac{1}{b^2}}.$$ Together with the fact that $g(x)$ is entire analytic in $\log A$, it suggests that the space $\cA^\oo(\cA_q)$ actually includes ``$\cA^\oo$ functions" on the space of the modular double $\cA_{q\til[q]}$ as well. See \cite{Ip} for further details.
\end{Rem}


\subsection{Multiplicative Unitary}\label{sec:corep:MU}

Given a $C^*$-algebra $\cA$, we will denote by $$M(\cA)=\{B\in\cB(\cH)|B\cA\sub\cA,\cA B\sub\cA\}$$ the multiplier algebra of $\cA$ viewed as a subset of
$\cB(\cH)$, and we let $\cK(\cH)\sub~\cB(\cH)$ denotes the compact operators acting on $\cH$.

Multiplicative unitaries are fundamental to the theory of quantum groups in the setting of $C^*$-algebras and von Neumann algebras. It is one single map that encodes all structure maps of a quantum group and of its generalized Pontryagin dual simultaneously \cite{Ti}. In particular, we can construct out of the multiplicative unitary a coproduct as well as a corepresentation of the quantum group. Here we recall the basic properties of the multiplicative unitary, and the construction of the multiplicative unitary defined in \cite{WZ} on the $ax+b$ quantum group $\cA$ (see also \cite{PW}).

\begin{Def}
A unitary element $W\in \cA\ox \cA$ is called a multiplicative unitary if it satisfies the pentagon equation
\Eq{\label{pentagon}W_{23}W_{12}=W_{12}W_{13}W_{23}.}

A multiplicative unitary provides us with the coproduct $\D:\cA\to M(\cA\ox \cA)$ given by
\Eq{\label{coproduct}\D(c)=W(c\ox 1)W^*.}
\end{Def}

\begin{Prop}
The pentagon equation \eqref{pentagon} implies the coassociativity of the coproduct defined by \eqref{coproduct}.
\end{Prop}

By representing the first copy of $\cA$ in $W$ as bounded operator on a Hilbert space $\cH$, we obtain a unitary
element $V\in M(\cK(\cH)\ox\cA)$ which represents a (right) corepresentation $\cH\to \cH\ox M(\cA)$. More precisely:

\begin{Prop}
The unitary element $V\in M(\cK(\cH)\ox\cA)$ satisfies
\Eq{(1\ox\D)V=V_{12}V_{13}} or formally
\Eq{(1\ox \D)\circ\Pi=(\Pi\ox 1)\circ \Pi }
where $\D$ is given by \eqref{coproduct} and $\Pi: \cH\to \cH\ox M(\cA)$ is given by
\Eq{\Pi(v):=V(v\ox 1).}
\end{Prop}

We will now focus on the case where $\cA$ is the quantum plane.
\begin{Prop}\cite{WZ}
Restricting $\cA_q$ to the quantum semigroup generated by positive self-adjoint elements $A,B\in\cA_q$ with $AB=q^2BA$
and coproduct
\Eq{\D(A)=A\ox A,\tab \D(B)=B\ox A+1\ox B} the multiplicative unitary $W$ is given by:
\Eq{\label{WW}W=V_\h(\log(\widehat{B}\ox s q\inv BA\inv))^* e^{\frac{i}{\hbar}\log\widehat{A}\ox\log A\inv}\in C_{\oo}(\cA_q)\ox C_{\oo}(\cA_q)}
where
$q=e^{-i\hbar}, \h=\frac{2\pi}{\hbar}$, the admissible pair $\what[B]=B\inv$ and $\what[A]=q AB\inv$, and $s\in\R_+$
is a constant. Note that in our case $\hbar=2\pi b^2$.

Here the special function $V_\h(z)$ is defined as
\Eq{V_\h(z)=\exp \left\{\frac{1}{2\pi i}\int_0^\oo \log(1+a^{-\h})\frac{da}{a+e^{-z}}\right\}.}

\end{Prop}

\begin{Rem}
Since we are using the ``transpose" of $\cA$ in \cite{WZ}, our $W$ is related to that in \cite{WZ} by
$$A=a\inv, B=-q ba\inv,$$ i.e. the antipode associated to $\cA$. Furthermore, the choice of $W$ is different from \cite{Ip}, where we used instead the GNS representations to obtain the canonical $W$.
\end{Rem}

\begin{Lem} $V_\h(z)$ and $G_b(z)$ are related by the following formula:
\Eq{V_{1/b^2}(z)=\ze_b G_b(\frac{Q}{2}-\frac{iz}{2\pi b})=\frac{1}{g_b(e^z)},}
and the complex conjugation is given by
\Eq{\label{VG}V_{1/b^2}(z)^*=\frac{\bar{\ze}_b}{G_b(\frac{Q}{2}-\frac{iz}{2\pi b})}=g_b(e^z),}
where we recall $\ze_b=e^{\frac{\pi i}{4}+\frac{\pi i}{12}(b^2+b^{-2})}$.
\end{Lem}

\begin{proof}
In order to rewrite $V_\h(z)$ in terms of $G_b(z)$, we pass to
Ruijsenaars's more general hyperbolic gamma function \eqref{RuiG}.
From \cite[(A.18)]{Ru}, we have
$$V_\h(z)=G(2\pi,2\pi/\h;z)\exp(-i\h z^2/8\pi -\frac{\pi i}{24}(\h+\frac{1}{\h}))$$
with $\h=\frac{2\pi}{\hbar}=\frac{1}{b^2}$.

Also using $$G(a_+;a_-;z)=G(1,\frac{a_+}{a_-};\frac{z}{a_-})$$ and
\eqref{TeschG}:
$$G(b,b\inv,z)=e^{\pi i z^2/2}e^{\pi i Q^2/8}G_b(\frac{Q}{2}-iz)$$
we obtain
$$V_{1/b^2}(z)=\ze_b G_b(\frac{Q}{2}-\frac{iz}{2\pi b})$$
and the complex conjugation
$$V_{1/b^2}(z)^*=\frac{\bar{\ze}_b}{G_b(\frac{Q}{2}-\frac{iz}{2\pi b})}.$$
\end{proof}

\subsection{Corepresentation of $C_\oo(\cA_q)$}\label{sec:corep:corep}
We can now define the coaction of the quantum space $C_\oo(\cA_q)$:
\begin{Thm}
For the choice $s=2\sin \pi b^2\in\R_+$, the multiplicative unitary $W$ defined in \eqref{WW} induces a unitary (right) coaction of the quantum space $C_\oo(\cA_q)$ on
$\cH=L^2(\R)$ by
$$\Pi: \cH\to \cH\ox M(C_\oo(\cA_q))$$
\Eq{\label{coaction1}
f(t)&\mapsto& F(x):=\int_{\R+i0}f(t)e^{\pi Q(t-x)}\frac{G_b(ix-it)}{(2\sin\pi b^2)^{ib\inv (x-t)}}A^{ib\inv x}B^{ib\inv (t-x)}dt, }
where $f(z)\in\cW$, and extends to $\cH$ by density.
\end{Thm}
\begin{Rem} The choice of $s$ is made so that we will obtain classical limit from $G_b$, as well as the necessary pairing in order to get the representation of $\cB_q$ in the next subsection.
\end{Rem}
\begin{proof}

The element $W$ can be reinterpreted as an element $$V\in M(\cK(\cH)\ox C_\oo(\cA_q))$$ by letting $\what[A],\what[B]$
act on $\cH=L^2(\R)$, hence giving rise to a corepresentation of $C_\oo(\cA_q)$. We start with $A=e^{2\pi bx},
B=e^{2\pi bp}$, so that the action is given by
\begin{eqnarray}
\what[A]&=& qAB\inv=qe^{2\pi bx}e^{-2\pi b p}=e^{2\pi b(x-p)},\\
\what[B]&=&B\inv = e^{-2\pi b p}.
\end{eqnarray}

However the action is nontrivial in the factor $$e^{\frac{i}{2\pi b^2}\log\what[A]\ox \log A}.$$

Hence we introduce a change of variables (of order 3) on $L^2(\R)$ given by Kashaev \cite{Ka2, FK}:

\Eq{\widetilde{\mathbf{A}}:=f(\a)\mapsto F(\b)=\int_\R e^{2\pi i \a\b}e^{\pi i \b^2-\pi i/12}f(\a)d\a} such that
$$\til[\mathbf{A}]\inv x\til[\mathbf{A}]=-p$$
$$\til[\mathbf{A}]\inv p\til[\mathbf{A}]=x-p.$$

Then the operator $\what[A]$ and $\what[B]$ becomes:
\begin{eqnarray}
\til[\mathbf{A}]\inv\what[A]\til[\mathbf{A}]&=&e^{-2\pi b x}\\
\til[\mathbf{A}]\inv\what[B]\til[\mathbf{A}]&=&e^{2\pi b(-x+p)}=qe^{-2\pi bx}e^{2\pi bp}
\end{eqnarray}

Hence given a function $f(x)\in L^2(\R)$, we have
\begin{eqnarray*}
e^{\frac{i}{2\pi b^2} \log \what[A]\ox \log A\inv}f(x)&=&e^{\frac{i}{2\pi b^2}(-2\pi b x) \log A\inv}f(x)\\
&=&f(x)A^{ib\inv x}
\end{eqnarray*}

Next we deal with the quantum dilogarithm function $V_\h(z)$.
From the Fourier transform formula (Lemma \ref{FT}), we found from \eqref{VG}
\Eq{V_{1/b^2}(z)^*=\int_{\R+i0}e^{ib\inv tz}e^{\pi Qt}G_b(-it)dt.}

Hence the operator $W$ acts as
\begin{eqnarray*}
(Wf)(x)&=&V_{1/b^2}(\log(\what[B]\ox q\inv s BA\inv))^* (f(x)A^{ib\inv x})\\
&=&\left(\int_{\R+i0}(\what[B]\ox (q\inv sBA\inv))^{ib\inv t}e^{\pi Q t}G_b(-it)dt \right)(f(x)A^{ib\inv x})\\
&=&\left(\int_{\R+i0}(\what[B]^{ib\inv t}\ox (q\inv sBA\inv)^{ib\inv t})e^{\pi Q t}G_b(-it)dt \right) (f(x)A^{ib\inv x}).
\end{eqnarray*}
Now $\what[B]$ formally acts as $qe^{-2\pi b x}f(x-ib)$, and by induction
$$\what[B]^n f(x)=q^{n^2}e^{-2\pi b nx}f(x-ibn).$$
Hence using functional calculus, $\what[B]^{ib\inv t}$ acts (as a unitary operator) by
$$\what[B]^{ib\inv t}\cdot f(x)=q^{-b^{-2}t^2}e^{-2\pi i t x}f(x+t)=e^{-\pi it^2-2\pi i t x}f(x+t).$$
Next $(sq\inv BA\inv)^{ib\inv t}$ can be split using the relation
$$(BA\inv)^n=q^{-n(n-1)}B^nA^{-n},$$ we have
$$(sq\inv BA\inv)^{ib\inv t}=s^{ib\inv t}q^{-ib\inv t}q^{b^{-2}t^2+ib\inv t}B^{ib\inv t}A^{-ib\inv t}=s^{ib\inv t}e^{\pi it^2}B^{ib\inv t}A^{-ib\inv t}.$$
Combining, we obtain
\begin{eqnarray*}
&&\int_{\R+i0}e^{-\pi it^2-2\pi i t x}e^{\pi Q t}q^{-2tx}G_b(-it) s^{ib\inv t}e^{\pi it^2}B^{ib\inv t}A^{-ib\inv t}A^{ib\inv (x+t)}f(x+t)dt\\
&=&\int_{\R+i0}e^{\pi Q t}e^{-2\pi itx}G_b(-it)s^{ib\inv t}B^{ib\inv t}f(x+t)A^{ib\inv x} dt\\
&=&\int_{\R+i0}f(x+t)e^{\pi Q t}G_b(-it)s^{ib\inv t}A^{ib\inv x}B^{ib\inv t} dt\\
&=&\int_{\R+i0}f(t)e^{\pi Q(t-x)}G_b(ix-it)s^{ib\inv (t-x)}A^{ib\inv x}B^{ib\inv (t-x)}dt\\
\end{eqnarray*}
Now by setting
$$s=2\sin\pi b^2=i(q\inv-q) \in\R_{>0}$$
we obtain
$$=\int_{\R+i0}f(t)e^{\pi Q(t-x)}\frac{G_b(ix-it)}{(2\sin\pi b^2)^{ib\inv (x-t)}}A^{ib\inv x}B^{ib\inv (t-x)}dt$$
as desired. We see that the integrand is bounded by the asymptotic properties of $G_b(ix)$.
\end{proof}

Starting from the coaction formula, we can also see that it is a corepresentation by manipulating the functional properties of the special function $G_b(x)$ directly:
\begin{Cor}\label{corepdelta} The coaction satisfies
$$(1\ox \D)\circ\Pi=(\Pi\ox 1)\circ \Pi$$
as a map from $\cH$ to $\cH\ox M(C_\oo(\cA_q)\ox C_\oo(\cA_q))$, where we recall that $\D$ is the coproduct of $\cA_q$
given by \Eqn[
\D(A)&=&A\ox A,\\
\D(B)&=&B\ox A+1\ox B ] and extend to $C_\oo(\cA_q)$ by
$$\D\left(\int_\R\int_{\R+i0}F(s,t)A^{is}B^{it}dsdt\right):=\int_\R\int_{\R+i0}F(s,t)\D(A^{is}B^{it})dsdt.$$
\end{Cor}
\begin{proof}
We check the corepresentation axioms formally.

First note that since $A,B$ are positive self-adjoint, the coproduct $\D(A)$ and $\D(B)$ is essentially self-adjoint,
hence it is well-defined. (We don't run into the problem of choosing self-adjoint extension as in \cite{WZ} since our
$B$ is positive.)

For notational convenience, without loss of generality we scale $b\inv x$ and $b\inv z$ to $x$ and $z$ respectively.
We need to calculate the coproduct $\D(A^{ix}B^{iz-ix})$:
\begin{eqnarray*}
\D(A^{ix}B^{iz-ix})&=&\D(A)^{ix}\D(B)^{iz-ix}\\
&=&(A\ox A)^{ix}(B\ox A+1\ox B)^{iz-ix}\\
&=&(A^{ix}\ox A^{ix}) B\int_\R d\t \veca{z-x\\ \t}_b (B\ox A)^{iz-ix-i\t}(1\ox B)^{i\t}\\
&=&b\int_{C} d\t \frac{G_b(ib\t-ibz+ibx)G_b(-ib\t)}{G_b(ibx-ibz)}(A^{ix}B^{iz-ix-i\t})\ox(A^{iz-i\t}B^{i\t})\\
&=&b\int_{C} d\t \frac{G_b(ib\t+ibx)G_b(-ibz-ib\t)}{G_b(ibx-ibz)}(A^{ix}B^{-ix-i\t})\ox(A^{-i\t}B^{i\t+iz}),\\
\end{eqnarray*}
where the contour $C$, as before, goes above the poles at $\t=-z$ and below the poles at $\t=-x$.
Hence we have

\begin{eqnarray*}
(1\ox\D)\circ\Pi f(x)&=&b^2\int_{\R+i0}\int_{C} f(z)\frac{G_b(ibx-ibz)e^{\pi Qb(z-x)}}{(2\sin\pi b^2)^{ix-iz}}\cdot\\
&&\frac{G_b(ib\t+ibx)G_b(-ibz-ib\t)}{G_b(ibx-ibz)}(A^{ix}B^{-ix-i\t})\ox(A^{-i\t}B^{i\t+iz}) d\t dz\\
&=&b^2\int_{\R+i0}\int_{C}  \frac{f(z)e^{\pi Qb(z-x)}}{(2\sin\pi b^2)^{ix-iz}}G_b(ib\t+ibx)G_b(-ibz-ib\t)\cdot\\
&&(A^{ix}B^{-ix-i\t})\ox(A^{-i\t}B^{i\t+iz}) d\t dz\\
&=&b^2\int_{\R-i0}\int_{\R+i0}  \frac{f(z)e^{\pi Qb(z-x)}}{(2\sin\pi b^2)^{ix-iz}}G_b(ib\t+ibx)G_b(-ibz-ib\t)\cdot\\
&&(A^{ix}B^{-ix-i\t})\ox(A^{-i\t}B^{i\t+iz}) dz d\t\\
&=&b^2\int_{\R+i0}\int_{\R+i0} \frac{f(z)e^{\pi Qb(z-x)}}{(2\sin\pi b^2)^{ix-iz}}G_b(ibx-ibw)G_b(ibw-ibz)\cdot\\
&& (A^{ix}B^{iw-ix})\ox (A^{iw}B^{iz-iw})dzdw,
\end{eqnarray*}
where in the change of order of integration, the contour is such that $\Im(z)>~\Im(\t)$ and $\Im(\t)<\Im(x)=0$, hence
the contour of $\t$ after interchanging is shifted to $\R-i0$. The decay properties of $G_b$ on $\t$ gaurantee the
change of order of integration.

Finally we have
\begin{eqnarray*}
(\Pi\ox 1)\circ \Pi f(x)&=&b^2\int_{\R+i0}\int_{\R+i0}  f(z)\frac{G_b(ibx-ibw)e^{\pi Qb(w-x)}}{(2\sin\pi b^2)^{ix-iw}}\cdot\\
&&\frac{G_b(ibw-ibz)e^{\pi Qb(z-w)}}{(2\sin\pi b^2)^{iw-iz}} (A^{ix}B^{iw-ix})\ox (A^{iw}B^{iz-iw})dzdw\\
&=&b^2\int_{\R+i0}\int_{\R+i0} \frac{f(z)e^{\pi Qb(z-x)}}{(2\sin\pi b^2)^{ix-iz}}G_b(ibx-ibw)G_b(ibw-ibz)\cdot\\
&& (A^{ix}B^{iw-ix})\ox (A^{iw}B^{iz-iw})dzdw\\
&=&(1\ox \D)\circ\Pi f(x).
\end{eqnarray*}
\end{proof}

After rewriting the coaction explicitly, the relationship between the quantum corepresentation and the classical
$ax+b$ group representation becomes clear:
\begin{Thm} \label{classicalcorep} Under the scaling by $x\to bx$, the limit of the coaction \eqref{coaction1} is precisely the representation $R_{+}$ of the $ax+b$ group.
Similarly, the coaction corresponding to $V^*$ is $R_-$.
\end{Thm}
\begin{proof} Under the scaling, the coaction becomes
$$\int_{\R+i0} \frac{G_b(ibx-ibz)e^{\pi Qb(z-x)}}{(2\sin\pi b^2)^{ix-iz}}A^{ix}B^{iz-ix}f(z)dz.$$
Using the limit formula \eqref{Glim} for $G_b(ibx)$, we have:
\Eqn[
&=&\int_{\R+i0} \frac{(2\pi b)G_b(ibx-ibz)e^{\pi b^2(z-x)}e^{\pi(z-x)}(e^{-\frac{\pi i}{2}})^{ix-iz}}{(-2i\sin\pi b^2)^{ix-iz}}A^{ix}B^{iz-ix}f(z)dz\\
&=& \frac{1}{2\pi}\left(\frac{\pi b^2}{\sin\pi b^2}\right)^{ix-iz}\int_{\R+i0} \frac{(2\pi b)G_b(ibx-ibz)}{(-2i\pi b^2)^{ix-iz}}e^{\pi b^2(z-x)}A^{ix}(-iB)^{iz-ix}f(z)dz\\
&\to&\frac{1}{2\pi}\int_{\R+i0} \G(ix-iz)A^{ix}(-iB)^{iz-ix}f(z)dz\\
&=&R_+f(x).
]
Taking the conjugate of the above formula and renaming the variables, we see that the coaction corresponding to $V^*$ is precisely $R_-$.

\end{proof}

\begin{Prop}\cite[(4.19)]{WZ} The space $C_{\oo}(\cA_q)$ can be recovered from the multiplicative unitary $V\in M(\cK(\cH)\ox\cA_q)$ by
\Eq{C_{\oo}(\cA_q)=\mbox{ norm closure of }\{(\w\ox1)V+(\w'\ox1)V^*|\w,\w'\in\cB(\cH)^*\}.}
\end{Prop}

Recall that $V$ corresponds to the representation $R_+$ and similarly $V^*$ corresponds to $R_-$. Therefore in the classical ``$ax+b$" group, the above translates to the fact that functions on $G_+$ is spanned by matrix coefficients
\Eq{\frac{1}{2\pi}\G(-iz)a^{iw}(-ib)^{iz},\tab\frac{1}{2\pi}\G(-iz)a^{iw}(ib)^{iz}}
corresponding to $V$ and $V^*$.

More explicitly, note that for functions on $G_+$ of the form
$$g(\log a)f(b),$$
where $g\in L^2(\R), f\in L^2([0,\oo))$ are analytic, we can write using Fourier transform as
$$\int_\R\int_\R \widehat{g}(s)a^{is}\widehat{f}(x)e^{ibx}dxds,$$
and then using formally the Mellin transform for $x>0$:
$$e^{\pm ibx}=\int_{\R+i0} \G(-it)(\pm ibx)^{it}dt,$$
we see that the function $F(a,b)$ can be rewritten as
\Eq{\int_\R\int_{\R+i0}\widehat{g}(s)\widetilde{f_+}(t) \G(-it)e^{\frac{\pi t}{2}}a^{is}b^{it} dtds+\int_\R\int_{\R+i0}\widehat{g}(s)\widetilde{f_-}(t) \G(-it)e^{-\frac{\pi t}{2}}a^{is}b^{it} dtds,}
where
$$\til[f_{\pm}](t)=\int_0^\oo \what[f](\pm x)x^{it}dx$$
is analytic in $0<\Im(t)<1$ and of rapid decay in this strip.

Therefore this proposition can be interpreted as a form of ``Peter-Weyl" theorem for the quantum group $\cA_q$, which
says that $C_\oo$ functions on $\cA_q$ is spanned continuously by matrix coefficients of its unitary
corepresentations. (For a similar result see \cite{Ip}, where a different multiplicative unitary $V$ is used).

\subsection{Pairing and Representation of $\cB_q$}\label{sec:corep:pairing}
Recall that given a non-degenerate Hopf pairing $\<\;,\;\>$, from a corepresentation of a Hopf algebra $\cA$, we can construct a corresponding representation of the dual Hopf algebra $\cB$ by
$$\cB\ox\cH\xto[1\ox\Pi]\cB\ox(\cH\ox\cA)=(\cB\ox\cA)\ox\cH\xto[\<\;,\;\>\ox Id]\cH.$$

Now let us define the pairing between the generators $(A,B)$ of $\cA_q$ and $(X,Y)$ of $\cB_q$ as follow:
\begin{Def} We define
$$\<A,X\>=q^{-2},\tab \<A,Y\>=0,$$
$$\<B,X\>=0,\tab \<B,Y\>=-i.$$
Then they satisfy the coproduct relations with
$$\<A^n B^m, X\>=\<A,X\>^n \d_{m0}=q^{-2n}\d_{m0},$$
$$\<A^n B^m, Y\>=\<A^n, 1\>\<B^m, Y\>=-i\d_{m1}.$$
\end{Def}

From this pairing, we can formally extend the pairing to elements in the subclass of $M(C_\oo(\cA_q))$.
Let $\cD$ denote the image of $\cW$ under the corepresentation $\Pi$ to $\cH\ox~ M(C_\oo(\cA_q))$.
Then $$\cD\sub BC(\R)\ox\cE \sub BC(\R)\ox\cF$$ where $BC(\R)$ are bounded continuous functions on $\R$;
$$\cE=\mbox{ Linear span of } \left\{A^{is}\int_{\R+i0} F(t)B^{it}dt\right\},$$
where $F(t)$ is the same as in the definition of $\cA_\oo(\cA_q)$: meromorphic with possible poles at $t=-in-im/b^2$, and of rapid decay along imaginary direction;
$$\cF= \mbox{ Linear span of} \left\{g(\log A)\int_{\R+i0} F(t)B^{it}dt\right\},$$
where $F(t)$ is as above, and $g(s)$ is a bounded function on $\R$ that can be analytically extended to $\Im s=-2\pi i b^2$. Then we define the pairing with $X$ and $Y$ by formally extracting the zeroth and first power of $B$ respectively. More precisely, we have

\begin{Def} We define $X,Y$ as elements in the dual space $\cF^*$ by
$$\<\frac{i}{2\pi }g(\log A)\int_{\R+i0} F(t) B^{it}dt , X\> = g(\log q^2)(\Res_{t=0}F(t)),$$
$$\<\frac{i}{2\pi }g(\log A)\int_{\R+i0} F(t) B^{it}dt, Y\> = -i(\Res_{t=-i}F(t)).$$
\end{Def}

\begin{Thm}The representation of $\cB_q$ on $\cW$ given by
$$\cB_q: \cW\xto[\Pi]BC(\R)\ox \cF \xto[1\ox\<-,\cB_q\>] BC(\R),$$ induced from the corepresentation \eqref{coaction1} under the above pairing is precisely
\Eqn[
X\cdot f(x)&=&e^{2\pi b x}f(x),\\
Y\cdot f(x)&=&f(x-ib)=e^{2\pi b p}f(x),
] which is the Fourier transformed action of \eqref{HKX}-\eqref{HKY} defined in \cite{FK}.
\end{Thm}
Note that the image of $\cW$ is actually preserved in $\cW\sub BC(\R)$.
\begin{proof}
Applying the pairing, and introducing the scaling of $b$ in $dz$, we obtain for any $f(x)\in\cW$:
\begin{eqnarray*}
&&\left\<\int_{\R+i0} f(z)\frac{G_b(ix-iz)e^{\pi Q(z-x)}}{(2\sin\pi b^2)^{ib\inv(x-z)}}A^{ib\inv x}B^{ib\inv(z-x)}dz, X\right\>\\
&&\mbox{changing $z$ to $bz+x$: }\\
&=&\left\<\int_{\R+i0} f(bz+x)\frac{bG_b(-ibz)e^{\pi Qb z}}{(2\sin\pi b^2)^{-iz}}A^{ib\inv x}B^{iz}dz, X\right\>\\
&=&(-2\pi i)f(x)q^{-2(ib\inv x)}b(\Res_{z=0}G_b(-ibz))\\
&=&e^{2\pi b x}f(x),
\end{eqnarray*}
since $\dis\lim_{x\to 0}xG_b(x)=\frac{1}{2\pi}$, hence $\Res_{z=0}G_b(-ibz)=\frac{1}{-2\pi i b}$.

So the action for $X$ is
$$X\cdot f(x)=e^{2\pi b x}f(x).$$

For the action of $Y$ we have
\begin{eqnarray*}
&&\left\<\int_{\R+i0} f(z)\frac{G_b(ix-iz)e^{\pi Q(z-x)}}{(2\sin \pi b^2)^{ib\inv(x-z)}}A^{ib\inv x}B^{ib\inv(z-x)}dz, Y\right\>\\
&=&\left\<\int_{\R+i0} f(bz+x)\frac{bG_b(-ibz)e^{\pi z(1+b^2)}}{(2\sin \pi b^2)^{-iz}}A^{ib\inv x}B^{iz}dz, Y\right\>\\
&=&(-i)(-2\pi i)f(x-ib)b(-q\inv) i(q\inv-q)(\Res_{z=-i}G_b(-ibz))\\
&=&f(x-ib),
\end{eqnarray*}
where
\begin{eqnarray*}
\Res_{z=-i}G_b(-ibz)&=&\lim_{z\to -i}(z+i)G_b(-ibz)\\
&=&\lim_{z\to0} zG_b(-ibz-b)\\
&=&\lim_{z\to0} z\frac{G_b(-ibz)}{1-e^{2\pi i b (-ibz-b)}}\\
&=&\frac{1}{-2\pi i b}\frac{1}{1-e^{-2\pi i b^2}}\\
&=&\frac{1}{-2\pi i b}\frac{1}{1-q^{-2}}.
\end{eqnarray*}
So the action for $Y$ is $$Y\cdot f(x)=f(x-ib)$$ or $Y=e^{2\pi b p}$.
\end{proof}

\begin{Rem} If we choose to work with $R_-$, then under the pairing we will get instead
$X=e^{2\pi b x}$ and $Y=-e^{2\pi b p}$, another representation for $\cB_q$ by negative operator $Y$.
\end{Rem}


\end{document}